\documentclass[11pt]{article}
\usepackage{enumerate}
\usepackage{amsmath}
\usepackage{amsthm}
\usepackage{amsfonts}
\usepackage{amssymb}
\usepackage[numbers]{natbib}
\usepackage{xcolor}
\usepackage{hyperref}

\usepackage{units}
\usepackage{float}
\usepackage{tikz}

\usepackage{graphicx}

\newcommand*\samethanks[1][\value{footnote}]{\footnotemark[#1]}
\usepackage{mathtools}

\usepackage{dsfont} 

\usepackage{fullpage}

\usepackage[capitalise]{cleveref} 
\crefname{equation}{}{}
\crefname{enumi}{}{}

\def\vp#1{}
\renewcommand{\vp}[1]{\footnote{\textcolor{green!40!black}{\textbf{VP: }#1}}}

\newcommand{\keywords}[1]
{
	{\small\textbf{Keywords:} #1}
}

\newtheorem{question}{Question}
\newtheorem{corollary}[question]{Corollary}
\newtheorem{problem}[question]{Problem}

\newtheorem{theorem}[question]{Theorem}

\newtheorem{proposition}[question]{Proposition}
\newtheorem{lemma}[question]{Lemma}
\newtheorem{remark}[question]{Remark}
\newtheorem{claim}[question]{Claim}
\newtheorem{definition}[question]{Definition}
\newtheorem{construction}[question]{Construction}
\numberwithin{question}{section}
\numberwithin{equation}{section}

\title{Minimum degree conditions for rainbow triangles}
\author{Victor Falgas-Ravry\thanks{Institutionen f\"or Matematik och Matematisk Statistik, Ume{\aa} Universitet, Sweden. Emails: \texttt{victor.falgas-ravry}, \texttt{klas.markstrom}, \texttt{eero.raty} \texttt{@umu.se}. } \and Klas Markstr\"om\samethanks \and Eero R\"aty\samethanks}
\begin{document}
	\maketitle	
\begin{abstract}
Let $\mathbf{G}:=(G_1, G_2, G_3)$ be a triple of graphs on a common vertex set $V$ of size $n$. A rainbow triangle in $\mathbf{G}$ is a triple of edges $(e_1, e_2, e_3)$ with $e_i\in G_i$ for each $i$ and $\{e_1, e_2, e_3\}$ forming a triangle in $V$.  In this paper we consider the following question: what triples of minimum degree conditions $(\delta(G_1), \delta(G_2), \delta(G_3))$ guarantee the existence of a rainbow triangle? This may be seen as a minimum degree version of a problem of Aharoni, DeVos, de la Maza, Montejanos and \v{S}\'amal on density conditions for rainbow triangles, which was recently resolved by the authors. We establish that the extremal behaviour in the minimum degree setting differs strikingly from that seen in the density setting, with discrete jumps as opposed to continuous transitions. Our work leaves a number of natural questions open, which we discuss.
\end{abstract}	

\noindent \keywords{extremal graph theory, rainbow triangles, Gallai colourings, Mantel's theorem, min degree}

\section{Introduction}
One of the foundational results in extremal graph theory is Mantel's celebrated 1907 theorem~\cite{Mantel07}, which asserts that a triangle-free graph $G$ on $n$ vertices can have at most $\lfloor \frac{n^2}{4}\rfloor$ edges, with equality if and only if $G$ is (isomorphic to) the complete balanced bipartite graph $T_2(n)$. Mantel's theorem is equivalent to its minimum-degree variant, which states that a triangle-free on $n$ vertices has minimum degree at most $n/2$, which is tight (again, by considering $T_2(n)$). In this paper we will investigate a rainbow version of the latter result: given a triple of graphs $(G_1, G_2, G_3)$ on a common vertex set $V$, what minimum degree conditions on the $G_i$ guarantee the existence of a \emph{rainbow triangle}, that is, a triangle with edges $\{e_1, e_2, e_3\}$ such that $e_i \in G_i$ for each $i$?

\noindent We begin by introducing the requisite terminology and background for our problem.
\begin{definition}\label{def: template}[Colouring templates, colourings]
	An $r$-colouring template on $V$ is an $r$-tuple $\mathbf{G}^{(r)}=(G_1, G_2, \ldots , G_r)$, where each of the $G_i$ is a graph on $V$. Whenever $r$ is clear from context, we omit the superscript $r$ and write $\mathbf{G}$ for $\mathbf{G}^{(r)}$.  An $r$-coloured graph $(H, c)$ is a graph $H=(V(H), E(H))$ together with an $r$-colouring of its edges $c: \ E(H)\rightarrow \{1,2, \ldots, r\}$.  (Note that an $r$-coloured graph may be identified with an $r$-colouring template where the colour classes $G_i$, $1\leq i\leq r$, are pairwise edge-disjoint.)
\end{definition}  
	 \begin{definition}[Coloured and rainbow subgraphs]
	 	  Given an $r$-coloured graph $(H,c)$, we say that an $r$-colouring template $\mathbf{G}^{(r)}$ on a vertex set $V$ contains a copy of $(H,c)$ as a subgraph if there is an injection $f: \ V(H)\rightarrow  V$ such that for each edge $e=\{x,y\} \in E(H)$ we have $\{f(x),f(y)\}\in G_{c(e)}$.  Further, given  a graph $H$, we say that $\mathbf{G}$ contains a rainbow copy of $H$ if $\mathbf{G}$ contains $(H,c)$ for some $r$-colouring $c: \ E(H)\rightarrow \{1,2, \ldots, r\}$  assigning distinct colours to distinct edges.
\end{definition}
\noindent Gallai~\cite{Gallai67} initiated the study of $r$-colourings with no rainbow triangles, proving a structure theorem that was subsequently re-discovered and extended by a number of other researchers~\cite{CameronEdmonds97,GyarfasSimonyi04}; $r$-coloured graphs and $r$-colouring templates containing no rainbow triangles are accordingly referred to as \emph{Gallai colourings} and \emph{Gallai colouring templates} respectively.

One of the first rainbow variants of Mantel's theorem was obtained by Keevash, Saks, Sudakov and Verstra\"etewho considered the following problem: how large does the arithmetic mean of the size of the colour classes $G_1, G_2, \ldots G_r$ have to be to guarantee the existence of a rainbow $H$ in an $r$-colouring template? As a special case of more general results, they showed~\cite[theorem 1.2]{KeevashSaksSudakovVerstraete04}  that when $H$ is a triangle, it is best to take $G_i=T_2(n), \forall i$ when $r \geq 4$, while for $r=3$ it is best to take $G_1=G_2=K_n$ and take $G_3$ to be the empty graph.  As the latter construction features  an empty colour class, it is natural to ask what happens in the $r=3$ case if all three of the colour classes  are large. This was done by  Diwan and Mubayi~\cite{DiwanMubayi06}, who asked: what is the least $\alpha>0$ such that for all $n$ sufficiently large, every $3$-colouring template $\mathbf{G}$ on an $n$-set $V$ with $\min\{\vert E(G_i)\vert: \ 1\leq i\leq 3\}>\alpha n^2$ contains a rainbow triangle? Their question was answered by Aharoni, DeVos, de la Maza, Montejanos and \v{S}\'amal~\cite[Theorem 1.2]{AharoniDeVosdelaMazaMontejanoSamalin20}, who determined the optimal value for $\alpha$. Aharoni \textit{et al}~\cite[Problem 1.3]{AharoniDeVosdelaMazaMontejanoSamalin20} suggested the more general problem of determining  which triples of edge densities $(\alpha_1, \alpha_2, \alpha_3)$ guarantee the existence  a rainbow triangle, where $\alpha_i$ denotes the edge-density of $G_i$. Later Frankl, Gy\"ori, He, Lv, Salia, Tompkins, Varga and Zhu~\cite{FranklGyoriHeLvSaliaTompkins22} considered  the related problem of determining the maximum of the geometric mean of the sizes of the colour classes  $\left(e(G_1)e(G_2)e(G_3)\right)^{1/3}$ over all $n$-vertex Gallai $3$-colouring templates $\mathbf{G}$, and provided a certain construction which they conjectured to be asymptotically extremal for this problem.

In~\cite{FRMarkstromRaty22} the authors resolved the question of Aharoni \textit{et al}, and, as a by-product, proved the aforementioned conjecture of Frankl \textit{et al}. However as we noted in that paper, the two families of extremal examples of Gallai colouring templates for the various triples of critical densities $(\alpha_1, \alpha_2, \alpha_3)$ given in~\cite[Theorem 1.13]{FRMarkstromRaty22} have as a common feature that at least one of the colour classes $G_i$, $i\in \{1,2,3\}$ contains isolated vertices. Motivated by this observation, we investigate in this paper the minimum degree variant of Aharoni \textit{et al}'s problem. As we show, the extremal behaviour in the minimum degree setting differs starkly from that seen in the edge-density setting. In particular, unlike what holds for Mantel's classical theorem, whose edge-density and minimum-degree versions are essentially equivalent,  determining edge-density and minimum-degree conditions for rainbow triangles are two genuinely different problems. To set our result in its proper context, we first make the following trivial observation.
\begin{proposition}\label{prop: trivial min degree bound}
Let $\mathbf{G}$ be a $3$-colouring template on $n$ vertices. Suppose $\delta(G_3)>0$. If $\delta(G_1)+\Delta(G_2)> n$, then $\mathbf{G}$ contains a rainbow triangle.
\end{proposition}
\begin{proof}
Suppose $\delta(G_3)>0$ and $\delta(G_1)+\Delta(G_2)>n$. Let $v$ be such that $d_{G_2}(v)=\Delta(G_2)$, and let $\{v,v'\}\in E(G_3)$. By the pigeonhole principle, there exists $v''$ such that $\{v,v''\}\in E(G_2)$ and $\{v', v''\}\in E(G_1)$. Thus $\{v,v',v''\}$ is a rainbow triangle in $\mathbf{G}$.
\end{proof}
\noindent
In particular, Proposition~\ref{prop: trivial min degree bound} implies that if $\mathbf{G}$ is a Gallai $3$-colouring template on $n$ vertices with $\delta(G_i)>0$ for $i\in \{2,3\}$ and $\delta(G_1)> \left(1-\frac{1}{r}\right)n$, then $\min\left(\delta(G_2), \delta(G_3)\right)<\frac{n}{r}$. Our main result is a substantial improvement of this trivial  bound: we show that in fact $\delta(G_2) + \delta(G_3) \leq \frac{2n}{r+1}$. As a consequence, it follows that in $n$-vertex Gallai $3$-colouring templates, the maximum value of $\min\left\{\delta(G_2), \delta(G_3)\right\}$ that can be attained jumps down from $\frac{n}{r}$ to $\frac{n}{r+1}$ when $\delta_1(G)$ increases from $n-\lceil\frac{n}{r}\rceil$ to $\lceil n-\frac{n}{r}\rceil+1$. This stands in sharp contrast to the more continuous behaviour seen in the edge-density case~\cite[Theorem 1.13]{FRMarkstromRaty22}.
\begin{theorem}\label{theorem: forcing min degrees}
Let $\mathbf{G}$ be a $3$-colouring template on $n$ vertices, and let $r\in \mathbb{N}$. If $\mathbf{G}$ contains no rainbow triangle and satisfies $\delta(G_1)>\left(1-\frac{1}{r}\right)n$ and $\delta(G_2)\geq \delta(G_3)>0$, then  $\delta(G_2) + \delta(G_3) \leq \frac{2n}{r+1}$. In particular, we must have $\delta(G_3) \leq \frac{n}{r+1}$. 
\end{theorem}

\noindent The upper bounds on the sum $\delta(G_2) + \delta(G_3)$ and on the value of $\delta(G_3)$ in Theorem~\ref{theorem: forcing min degrees} are easily seen to be tight up to an additive term of $1$, by considering the following Tur\'an like construction.
\begin{construction}
Given $r\in \mathbb{N}$, we let $\mathbf{T}=\mathbf{T}(r,n)$ denote the $3$-colouring template in which $T_1$ is a complete, balanced $r$-partite graph on $n$ vertices and $T_2=T_3$ is the union of disjoint complete graphs on each of the $r$ sets in the $r$-partition of $T_1$. 
\end{construction}
\noindent
Clearly, for every $r\geq 1$, $\mathbf{T}$ is a Gallai colouring template, and the minimum degrees of the colour classes in $\mathbf{T}$ meet the trivial bounds from Proposition~\ref{prop: trivial min degree bound}. In particular note that $\delta(T_1(r+1,n))>(1-\frac{1}{r})n$ and $\delta(T_2(r+1,n))=\delta(T_3(r+1,n))=\lfloor \frac{n}{r+1}\rfloor -1$, while $\delta(T_1(r,n))=\lfloor(1-\frac{1}{r})n\rfloor$ and $\delta(T_2(r,n))=\delta(T_3(r,n))=\lfloor\frac{n}{r}\rfloor-1> \frac{n}{r+1}$. Many other near extremal constructions are possible, however. Indeed, consider the following general family of constructions. 

\begin{construction}\label{construction: general}
	Let $r\geq 3$, and let $P$ be a graph on $[r]$ consisting of a disjoint union of isolated vertices, edges and copies of $K_4$. Let $c: \ E(P)\rightarrow \{0, 2,3\}$ be an arbitrary proper edge colouring of the edges of $P$ (i.e.\ a colouring such that adjacent edges are assigned distinct colours).

	Now define a colouring template $\mathbf{C_{P,c}}=\mathbf{C_{P,c}}(n)$ from $P$ and $c$ as follows. Let $\sqcup_{i=1}^r V_i$ be a balanced partition of $V=[n]$ into $r$ disjoint sets. Set 
	$(C_{P,c})_2$ to be the union of (i) all $V_i^{(2)}$  for which $i$ is an isolated vertex in $P$ or a vertex in a component of size $2$ whose edge is in colour $3$, and
	(ii) all $(V_{i}, V_j)^{(2)}$ for which $ij$ is an edge of $P$ with $c(ij)=2$ or $ij$ is an edge of a component of size $2$ whose edge is in colour $0$. Similarly, 	set $(C_{P,c})_3$ to be the union of (i) all $V_i^{(2)}$  for which $i$ is an isolated vertex in $P$ or a vertex in a component of size $2$ whose edge is in colour $2$, and (ii) all $(V_{i}, V_j)^{(2)}$ for which $ij$ is an edge of $P$ with $c(ij)=3$ or an edge of a component of size $2$ whose edge is in colour $0$. Finally, set $(C_{P,c})_1$ to be the union of (i) all $V_i^{(2)}$ for which $i$ is not an isolated vertex in $P$, and (ii) all $(V_i, V_j)^{(2)}$ such that $ij$ is a non-edge of $P$ or an edge of $P$ with $c(ij)\in\{2,3\}$, or an edge in a component of size $2$ whose edge is in colour $0$.
\end{construction}
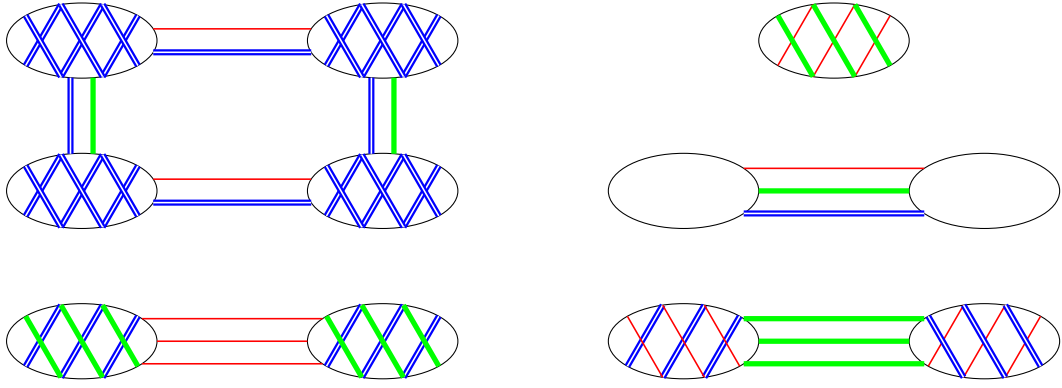
\begin{figure}\label{figure: P construction}\centering

		\begin{tikzpicture}
	\draw [line width = 0.0 mm, white] (0, 5.5) -- (1,5.5);

	
	\draw (0,4) ellipse (1cm and 0.5cm);
	\draw (4,4) ellipse (1cm and 0.5cm);
	
	\draw (0,2) ellipse (1cm and 0.5cm);
	\draw (4,2) ellipse (1cm and 0.5cm);

	\draw [line width = 0.3 mm, blue,double] (-0.7467, 4 -0.3326) -- (-0.2773,4+ 0.48);
	\draw [line width = 0.3 mm, blue,double] (-0.2773,4 -0.48) -- (0.2773,4+ 0.48);
	\draw [line width = 0.3 mm, blue,double] (0.2773,4 -0.48) -- (0.7467,4 + 0.3326);

	\draw [line width = 0.3 mm, blue,double] (-0.7467, 4 + 0.3326) -- (-0.2773,4-0.48);
	\draw [line width = 0.3 mm, blue,double] (-0.2773,4+ 0.48) -- (0.2773,4-0.48);
	\draw [line width = 0.3 mm, blue,double] (0.2773,4+0.48) -- (0.7467,4-0.3326);

	\draw [line width = 0.3 mm, blue,double] (-0.7467, 2 -0.3326) -- (-0.2773,2+ 0.48);
	\draw [line width = 0.3 mm, blue,double] (-0.2773,2 -0.48) -- (0.2773,2+ 0.48);
	\draw [line width = 0.3 mm, blue,double] (0.2773,2 -0.48) -- (0.7467,2 + 0.3326);

	\draw [line width = 0.3 mm, blue,double] (-0.7467, 2 + 0.3326) -- (-0.2773,2-0.48);
	\draw [line width = 0.3 mm, blue,double] (-0.2773,2+ 0.48) -- (0.2773,2-0.48);
	\draw [line width = 0.3 mm, blue,double] (0.2773,2+0.48) -- (0.7467,2-0.3326);

	\draw [line width = 0.3 mm, blue,double] (4-0.7467, 4 -0.3326) -- (4-0.2773,4+ 0.48);
	\draw [line width = 0.3 mm, blue,double] (4-0.2773,4 -0.48) -- (4+0.2773,4+ 0.48);
	\draw [line width = 0.3 mm, blue,double] (4+0.2773,4 -0.48) -- (4+0.7467,4 + 0.3326);

	\draw [line width = 0.3 mm, blue,double] (4-0.7467, 4 + 0.3326) -- (4-0.2773,4-0.48);
	\draw [line width = 0.3 mm, blue,double] (4-0.2773,4+ 0.48) -- (4+0.2773,4-0.48);
	\draw [line width = 0.3 mm, blue,double] (4+0.2773,4+0.48) -- (4+0.7467,4-0.3326);

	\draw [line width = 0.3 mm, blue,double] (4-0.7467, 2 -0.3326) -- (4-0.2773,2+ 0.48);
	\draw [line width = 0.3 mm, blue,double] (4-0.2773,2 -0.48) -- (4+0.2773,2+ 0.48);
	\draw [line width = 0.3 mm, blue,double] (4+0.2773,2 -0.48) -- (4+0.7467,2 + 0.3326);

	\draw [line width = 0.3 mm, blue,double] (4-0.7467, 2 + 0.3326) -- (4-0.2773,2-0.48);
	\draw [line width = 0.3 mm, blue,double] (4-0.2773,2+ 0.48) -- (4+0.2773,2-0.48);
	\draw [line width = 0.3 mm, blue,double] (4+0.2773,2+0.48) -- (4+0.7467,2-0.3326);	

	\draw [line width = 0.7 mm, red,semithick] (0.95, 2+0.1561) -- (4- 0.95, 2+0.1561);
	\draw [line width = 0.3 mm, blue,double] (0.95, 2-0.1561) -- (4-0.95, 2-0.1561);
	
	\draw [line width = 0.7 mm, red,semithick] (0.95, 4+0.1561) -- (4-0.95, 4+0.1561);
	\draw [line width = 0.3 mm, blue,double] (0.95, 4-0.1561) -- (4-0.95, 4-0.1561);

	\draw [line width = 0.3 mm, blue,double] (-0.15, 4 - 0.4943) -- (-0.15, 2 + 0.4943);
	\draw [line width = 0.7 mm, green] (0.15, 4 - 0.4943) -- (0.15, 2 + 0.4943);
	
	\draw [line width = 0.3 mm, blue,double] (4-0.15, 4 - 0.4943) -- (4-0.15, 2 + 0.4943);
	\draw [line width = 0.7 mm, green] (4+0.15, 4 - 0.4943) -- (4+0.15, 2 + 0.4943);
	
	
	\draw (8,2) ellipse (1cm and 0.5cm);
	\draw (12,2) ellipse (1cm and 0.5cm);
	
	\draw [line width = 0.7 mm, red,semithick] (8+0.8, 2+0.3) -- (8+3.2, 2+0.3);
	\draw [line width = 0.3 mm, blue,double] (8+0.8, 2-0.3) -- (8+3.2,2 -0.3);
	\draw [line width = 0.7 mm, green] (8+1, 2) -- (8+3, 2);
	
	
	\draw(10,4) ellipse (1cm and 0.5cm);
	\draw [line width = 0.7 mm, red,semithick] (10-0.7467, 4-0.3326) -- (10-0.2773,4+0.48);
	\draw [line width = 0.7 mm, red,semithick] (10-0.2773,4-0.48) -- (10+0.2773,4+0.48);
	\draw [line width = 0.7 mm, red,semithick] (10+0.2773,4-0.48) -- (10+0.7467,4+0.3326);

	\draw [line width = 0.7 mm, green] (10-0.7467, 4+0.3326) -- (10-0.2773,4-0.48);
	\draw [line width = 0.7 mm, green] (10-0.2773,4+0.48) -- (10+0.2773,4-0.48);
	\draw [line width = 0.7 mm, green] (10+0.2773,4+0.48) -- (10+0.7467,4-0.3326);
	
	\draw (0,0) ellipse (1cm and 0.5cm);
	\draw (4,0) ellipse (1cm and 0.5cm);

	\draw [line width = 0.7 mm, red,semithick] (0.8, 0.3) -- (3.2, 0.3);
	\draw [line width = 0.7 mm, red,semithick] (0.8, -0.3) -- (3.2, -0.3);
	\draw [line width = 0.7 mm, red,semithick] (1, 0) -- (3, 0);
	
	\draw [line width = 0.3 mm, blue,double] (-0.7467, -0.3326) -- (-0.2773,0.48);
	\draw [line width = 0.3 mm, blue,double] (-0.2773,-0.48) -- (0.2773,0.48);
	\draw [line width = 0.3 mm, blue,double] (0.2773,-0.48) -- (0.7467,0.3326);

	\draw [line width = 0.7 mm, green] (-0.7467, 0.3326) -- (-0.2773,-0.48);
	\draw [line width = 0.7 mm, green] (-0.2773,0.48) -- (0.2773,-0.48);
	\draw [line width = 0.7 mm, green] (0.2773,0.48) -- (0.7467,-0.3326);

	\draw [line width = 0.3 mm, blue,double] (4-0.7467, -0.3326) -- (4-0.2773,0.48);
	\draw [line width = 0.3 mm, blue,double] (4-0.2773,-0.48) -- (4+0.2773,0.48);
	\draw [line width = 0.3 mm, blue,double] (4+0.2773,-0.48) -- (4+0.7467,0.3326);

	\draw [line width = 0.7 mm, green] (4-0.7467, 0.3326) -- (4-0.2773,-0.48);
	\draw [line width = 0.7 mm, green] (4-0.2773,0.48) -- (4+0.2773,-0.48);
	\draw [line width = 0.7 mm, green] (4+0.2773,0.48) -- (4+0.7467,-0.3326);
	

	
	\draw (8,0) ellipse (1cm and 0.5cm);
	\draw (12,0) ellipse (1cm and 0.5cm);

	\draw [line width = 0.7 mm, green] (8 + 0.8, 0.3) -- (8 + 3.2, 0.3);
	\draw [line width = 0.7 mm, green] (8 + 0.8, -0.3) -- (8 + 3.2, -0.3);
	\draw [line width = 0.7 mm, green] (8 + 1, 0) -- (8 + 3, 0);

	\draw [line width = 0.3 mm, blue,double] (8-0.7467, -0.3326) -- (8-0.2773,0.48);
	\draw [line width = 0.3 mm, blue,double] (8-0.2773,-0.48) -- (8+0.2773,0.48);
	\draw [line width = 0.3 mm, blue,double] (8+0.2773,-0.48) -- (8+0.7467,0.3326);

	\draw [line width = 0.7 mm, red,semithick] (8-0.7467, 0.3326) -- (8-0.2773,-0.48);
	\draw [line width = 0.7 mm, red,semithick] (8-0.2773,0.48) -- (8+0.2773,-0.48);
	\draw [line width = 0.7 mm, red,semithick] (8+0.2773,0.48) -- (8+0.7467,-0.3326);
	
	\draw [line width = 0.7 mm, red,semithick] (8+4-0.7467, -0.3326) -- (8+4-0.2773,0.48);
	\draw [line width = 0.7 mm, red,semithick] (8+4-0.2773,-0.48) -- (8+4+0.2773,0.48);
	\draw [line width = 0.7 mm, red,semithick] (8+4+0.2773,-0.48) -- (8+4+0.7467,0.3326);

	\draw [line width = 0.3 mm, blue,double] (8+4-0.7467, 0.3326) -- (8+4-0.2773,-0.48);
	\draw [line width = 0.3 mm, blue,double] (8+4-0.2773,0.48) -- (8+4+0.2773,-0.48);
	\draw [line width = 0.3 mm, blue,double] (8+4+0.2773,0.48) -- (8+4+0.7467,-0.3326);
	
	\end{tikzpicture}

	\caption{The building blocks used in Construction~\ref{construction: general}, with colours blue (doubled lines), red (thin lines) and green (thick lines) corresponding to edges in colours $1$, $2$ and $3$ respectively. The top left figure corresponds to a $K_4$ in $P$, the top right figure corresponds to an isolated vertex in $P$, the middle right figure to a component of size $2$ in $P$ whose edge is assigned the colour $0$, the bottom right figure to a component of size $2$ whose edge is assigned the colour $3$, and the bottom left figure to a component of size $2$ whose edge is assigned the colour $2$. A $\mathbf{C_P}$ colouring template is built by combining vertex-disjoint copies of these building blocks, and connecting distinct building blocks by edges in colour $1$. }
\end{figure}	
\noindent See Figure~\ref{figure: P construction} for a picture of the building blocks used to construct $\mathbf{C_{P,c}}$. It is easily checked that for any graph $P$ on $r\geq 2$ vertices consisting of vertex-disjoint isolated vertices, edges and copies of $K_4$, $\mathbf{C_{P,c}}$ is a Gallai colouring template with 
\begin{align*}
n-\left\lceil\frac{n}{r}\right\rceil-1\leq \delta((C_{P,c})_1) \leq n-\frac{n}{r}, && \left\lfloor\frac{n}{r}\right\rfloor -1 \leq \min\left( \delta\left({(C_{P,c})}_2\right),\delta\left({(C_{P,c})}_3\right)\right)\leq \frac{n}{r}.
\end{align*} 
In particular, we do not have stability for Theorem~\ref{theorem: forcing min degrees} (or at least no unique stable pattern) --- a somewhat unusual feature in extremal graph theory.

We now move on to some notation which is required for the proofs, followed by proofs,  and at the end of the paper we give some further discussion and open problems.

\subsection{Notation}
We write $[n]:=\{1,2,\ldots n\}$,  $S^{(2)}:=\{\{s,s'\}:\  s,s' \in S, s\neq s'\}$ and $(S,T)^{(2)}:=\{\{s,t\}: \ s\in S, t\in T\}$. Where convenient, we identify $G_i$ with its edge-set  $E(G_i)$.  We also write $xy$ for $\{x,y\}$. We use $G_i[X]$ and $G_i[X,Y]$ as a notation for the subgraph of $G_i$ induced by the vertex-set $X$ and for the bipartite subgraph of $G_i$ induced by the bipartition $X\sqcup Y$ respectively. Throughout the remainder of the paper, we shall use $\vert G_i\vert$, $\vert G_i[X]\vert$ and $\vert G_i[X,Y]\vert$  as shorthands for $\vert E(G_i)\vert$, $\vert E(G_i[X])\vert$ and $\vert E(G_i[X,Y])\vert$ respectively.

We use standard graph theoretic terminology. In particular, $N_G(v)$ denotes the neighbourhood of the vertex $v$ in the graph $G$, and $d_G(v)$ its degree. We write $\delta(G):=\min\{ d_G(v): \ v\in V(G)\}$ for the minimum degree of $G$ and $\Delta(G)$ for its maximum degree. Given a $3$-colouring template $\mathbf{G}$ on a set $V$, we call a pair $xy\in V^{(2)}$ a \emph{rainbow edge} if $xy\in \bigcap_{i=1}^3G_i$. Further, we call a pair $xy$ which is contained in at least two of the colour classes $G_1, G_2, G_3$ a \emph{bi-chromatic edge}.

\section{Minimum colour-degree conditions for rainbow triangles}\label{section: min degree}
\subsection{Proof strategy}\label{subsection: proof strategy mindegree}
\begin{proof}[Proof of Theorem~\ref{theorem: forcing min degrees}]
We shall prove Theorem~\ref{theorem: forcing min degrees} by contradiction: suppose its statement is false. Let $r$ be the least natural number for which it fails to be true, and let $\mathbf{G}$ be a Gallai colouring template on $N$ vertices that provides a counterexample to Theorem~\ref{theorem: forcing min degrees} for that value of $r$.

It is easy to see that $r\geq 2$. Indeed, for any Gallai $3$-colouring template $\mathbf{G}$ on $n$ vertices with $\delta(G_1)>0$, Proposition~\ref{prop: trivial min degree bound} (with colours $1$ and $3$ interchanged), implies that $\delta(G_2)+\delta (G_3)< n$.

If $r\geq 3$, then we shall exploit the minimum degree conditions $\delta(G_1)>(1-\frac{1}{r})N$ and $\delta(G_2)+ \delta(G_3)>\frac{2N}{r+1}$ together with the absence of rainbow triangles to infer structural information about $\mathbf{G}$, from which we will eventually derive a contradiction. More specifically, we will first show that there are no bi-chromatic edges in colours $23$ (Lemma~\ref{lem: mindeg, no 23 edge}), and that $G_2$ is a bipartite graph (Lemma~\ref{lem: mindeg, G2 bipartite})  while $G_3$ is well-behaved with respect to $G_2$'s bipartition (Lemma~\ref{lemma: G3 structure}). From there, we shall be able to derive a contradiction and show $\mathbf{G}$ cannot both satisfy our minimum degree conditions and fail to contain a rainbow triangle.

Our approach to the case $r=2$ is similar in spirit (we rule out bi-chromatic edges in colour $23$ in Lemma~\ref{lem: mindeg, no 23 edge r=2} and show one of $G_2$, $G_3$ must be bipartite in Lemma~\ref{lemma: G_3 bipartite r=2}, before deriving a contradiction); however for $r=2$, the key inequality
\begin{align}\label{eq: key ineq r at least 3}
\delta(G_1)+ 2\min(\delta(G_2),\delta(G_3))>N\end{align}
which drives many of the proofs in the $r\geq3$ case fails. Indeed, since $\delta(G_3)>0$ and $\delta(G_1)>\frac{r-1}{r}N$, Proposition~\ref{prop: trivial min degree bound} implies that $\delta(G_2)\leq \Delta(G_2)\leq  N-\delta(G_1)<\frac{N}{r}$. Thus 
\begin{align}\label{eq: (r-1)N/(r(r+1)) bound on delta G3}\delta(G_3)\geq \frac{2N}{r+1}-\delta(G_2)>\frac{r-1}{r(r+1)}N,
\end{align} which is greater or equal to $\frac{N}{2r}$ (and thus implies~\eqref{eq: key ineq r at least 3}) if and only if $r\geq 3$. If $r\geq 3$, thanks to~\eqref{eq: key ineq r at least 3} we are able to treat the colours $2$ and $3$ interchangeably. This is no longer true when $r=2$, and we must come up with alternative and much longer arguments that are specially tailored to that setting and do not rely on~\eqref{eq: key ineq r at least 3}. We deal with the $r\geq 3$ and $r=2$ cases in Section~\ref{subsection: r>2 case} and Section~\ref{subsection: proof of r=2} respectively, proving the following theorems: 
\begin{theorem}\label{theorem: r =2 case}
Suppose $\mathbf{G}$ is a Gallai $3$-colouring template on $N$ vertices with $\delta(G_1)> \frac{N}{2}$ and $\delta(G_2)\geq  \delta(G_3)>0$. Then $\delta(G_2)+\delta(G_3)\leq \frac{2N}{3}$.
\end{theorem}
\begin{theorem}\label{theorem: r >2 case}
	Suppose $\mathbf{G}$ is a Gallai $3$-colouring template on $N$ vertices with $\delta(G_1)> \frac{(r-1)N}{r}$, for some $r\geq 3$, $\delta(G_2)>0$ and $\delta(G_3)>0$. Then $\delta(G_2)+\delta(G_3)\leq \frac{2N}{r+1}$.
\end{theorem}
\noindent Theorem~\ref{theorem: forcing min degrees} follows immediately from these two results together with Proposition~\ref{prop: trivial min degree bound}. \end{proof}

\subsection{Proof of the case $r\geq 3$}\label{subsection: r>2 case}
\begin{proof}[Proof of Theorem~\ref{theorem: r >2 case}]
Having outlined our proof strategy in the subsection above, we now carry it out. By Proposition~\ref{prop: trivial min degree bound} and Theorem~\ref{theorem: r =2 case}, we may assume the statement of Theorem~\ref{theorem: forcing min degrees} is true for $r=1$ and $r=2$. Suppose Theorem~\ref{theorem: r >2 case} is false, and let $r\geq 3$ be the least natural number for which it fails. We let $\mathbf{G}$ be a putative counterexample on $N$ vertices for this minimum value of $r$, and we denote its vertex set by $V$.

Note that in this proof (as in the statement of Theorem~\ref{theorem: r >2 case}), the colours $2$ and $3$ play a symmetrical role. Given a colour $c\in \{2,3\}$, we shall denote by $\overline{c}$ the complementary colour, $\{\overline{c}\}:=\{2,3\}\setminus \{c\}$. We observe that, by applying Proposition~\ref{prop: trivial min degree bound} as in the proof of~\eqref{eq: (r-1)N/(r(r+1)) bound on delta G3}, the bounds
\begin{align}\label{eq: ub and lb bound on delta(G_i)}
\frac{N}{r}>\Delta(G_c)\geq \delta(G_c) \geq \frac{2N}{r+1}-\Delta(G_{\overline{c}}) > \frac{r-1}{r(r+1)}N
\end{align}
hold for each $c\in \{2,3\}$. Further, the bound $\delta(G_1)+2\delta(G_c)> \delta(G_1)+2\frac{r-1}{r(r+1)}N>N$ holds (since $r\geq 3$).

We begin by showing that there is no non-trivial bipartition of $V$ with no edge in colour $2$ or $3$ going across it; this  simple fact will be used several times in the argument.
\begin{definition}[$1$-cut]
A bipartition $V=A\sqcup B$ is a $1$-cut if $\vert G_2[A,B]\vert =\vert G_3[A,B]\vert=0$, i.e.\ there are no edges in colours $2$ or $3$ across the partition. A $1$-cut is non-trivial if both $A$ and $B$ are non-empty.
\end{definition}
\begin{lemma}\label{lem:1-cut} 
	The Gallai colouring template $\mathbf{G}$ does not contain a non-trivial $1$-cut.
\end{lemma}
\begin{proof}
Suppose $A\sqcup B$ is a non-trivial $1$-cut with $\left|A\right|=\frac{\alpha N}{r+1}$ and $\left|B\right|=\frac{\beta N}{r+1}$,
	where $\alpha+\beta=r+1$. Our minimum degree condition $\delta(G_2) + \delta(G_3)>\frac{2N}{r+1}$ together with the non-triviality of the $1$-cut $A\sqcup B$ implies $1<\alpha,\beta<r$. Note that for any vertex $v\in A$, $d_{G_2[A]}(v)=d_{G_2}(v)\geq \delta(G_2)$ and $d_{G_3[A]}(v)=d_{G_3}(v)\geq \delta(G_3)$. Further, we have 
	\begin{align*}
	\frac{d_{G_1[A]}\left(v\right)}{\left|A\right|}\geq \frac{d_{G_1}(v)-(N-\vert A\vert)}{\vert A\vert }>\frac{\left(1-\frac{1}{r}\right)N-\left(N-\frac{\alpha N}{r+1}\right)}{\frac{\alpha N}{r+1}}=1-\frac{r+1}{\alpha r}.
	\end{align*}
	
Suppose that $\alpha r\geq (r+1)\left(\lceil \alpha \rceil -1\right)$. The inequality above then tells us that $\delta(G_1[A])>\left(1-\frac{1}{\lceil \alpha \rceil -1}\right)\vert A\vert$. Since the statement of Theorem~\ref{theorem: forcing min degrees} holds for $r'=\lceil \alpha\rceil-1<r$ and since $\mathbf{G}[A]$ is a  Gallai colouring template, this implies that 
	\[ \delta(G_2) + \delta(G_3) \leq \delta(G_2[A]) + \delta(G_3[A]) \leq \frac{2\vert A\vert}{\lceil\alpha \rceil }\leq \frac{2N}{r+1},\] 
contradicting our assumption on $\mathbf{G}$. Thus we must have $\alpha r< (r+1)\left(\lceil \alpha \rceil -1\right)$ and, by a symmetric argument for  $v\in B$, $\beta r< (r+1)\left(\lceil \beta \rceil -1\right)$. Summing these two inequalities, and using $\alpha+\beta=r+1$, we get a contradiction:
\begin{align*}
(r+1)r=(\alpha+\beta) r<(r+1)\left(\lceil \alpha\rceil +\lceil \beta \rceil -2\right) \leq (r+1)r.
\end{align*}
\end{proof}
\noindent Next we prove another useful result, namely that whenever a vertex is incident with at least one bi-chromatic edge in colours $ij$, it is incident to many such edges. To make this more precise, we define the \emph{minimum positive degree} $\delta^+(G)$ of a non-empty graph $G$ to be the minimum of $d_G(v)$ over all non-isolated vertices $v$.
\begin{lemma}\label{lem:  intersections, if large, are non-empty} 
	For every pair $i,j$ of distinct colours from $\{1,2,3\}$, either $\delta^+(G_i\cap G_j)>\frac{(r-1)}{r(r+1)}N$ or $G_i\cap G_j =\emptyset$.
\end{lemma}
\begin{proof}
Suppose $u$ is a vertex with $d_{G_i\cap G_j}(u)>0$. Let $\{k\}=\{1,2,3\}\setminus \{i,j\}$, and let $v$ be any vertex in $N_{G_i\cap G_j}(u)$. Since $\mathbf{G}$ is a  Gallai colouring template, $N_{G_k}(v)$ is disjoint from $N_{G_i\cup G_j}(u)$. In particular,
\begin{align*}
d_{G_i\cap G_j}(u)= \vert N_{G_i}(u)\vert + \vert N_{G_j}(u)\vert -\vert N_{G_i\cup G_j}(u)\vert &\geq \delta(G_i)+\delta(G_j) - \left(N- \vert N_{G_k}(v)\vert\right)\\
&\geq \delta(G_1)+\delta(G_2)+\delta(G_3)-N> \frac{(r-1)}{r(r+1)}N.
\end{align*}
\end{proof}
\begin{remark}\label{remark: also works for r=3}
	Both Lemma~\ref{lem:1-cut} and Lemma~\ref{lem:  intersections, if large, are non-empty} also hold for $r=2$ (since as we observed in the previous subsection, the case $r=1$ of Theorem~\ref{theorem: forcing min degrees} follows from Proposition~\ref{prop: trivial min degree bound}).
\end{remark}
\noindent We are now in a position to prove our first key structural lemma, namely that there are no bi-chromatic edges in colours $23$ in $\mathbf{G}$. 
\begin{lemma}\label{lem: mindeg, no 23 edge}
$G_2\cap G_3 =\emptyset$.	
\end{lemma}
\begin{proof}
Set $X$ to be the set of vertices in $V$ incident with a bi-chromatic edge in colours $23$ from the  graph $G_2\cap G_3$. Suppose for contradiction $X\neq \emptyset$. Let $Y$ denote the set of $y\in V\setminus X$ such that $y$ sends at least one edge in colour $2$ or $3$ to $X$. Finally, let $Z:= V\setminus (X\cup Y)$ denote the rest of the vertices in $V$. 
\begin{claim}\label{claim: Z is empty}
	$Z=\emptyset$
\end{claim}
\begin{proof}
Suppose this is not the case. Since $\mathbf{G}$ contains no non-trivial $1$-cut (by Lemma~\ref{lem:1-cut}), there exists $z\in Z$ sending an edge in a colour $c\in\{2,3\}$ to some vertex in $ X\cup Y$; further, by construction of $Y$, this vertex $y$ must lie in $Y$. By construction of $Y$ again, there exists $x\in X$ such that $xy\in G_{i}$ for some $i\in \{2,3\}$.

By construction of $X$ and Lemma~\ref{lem:  intersections, if large, are non-empty}, $N_{G_2\cap G_3}(x)$ is a subset of $X$ of size at least $\delta^+(G_2\cap G_3)>\frac{(r-1)}{r(r+1)}N$. By construction of $Z$ and~\eqref{eq: ub and lb bound on delta(G_i)}, $N_{G_{\overline{c}}}(z)$ is a subset of $Y\cup Z$ of size at least $\delta(G_{\overline{c}})>\frac{r-1}{r(r+1)}N$.  Finally, as $\mathbf{G}$ contains no rainbow triangle, both of these sets are disjoint from $N_{G_1}(y)$. Since
\begin{align*}
\vert N_{G_1}(y)\vert + \vert N_{G_{\overline{c}}}(z)\vert +\vert N_{G_2\cap G_3}(x)\vert > \delta(G_1)+\frac{2(r-1)}{r(r+1)N}>N,
\end{align*}
 (since $r\geq 3$), this yields a contradiction.	
\end{proof}
\begin{claim}\label{claim: X is not too small}
	$\vert X\vert \geq r$.
\end{claim}
\begin{proof}
Suppose $\vert X\vert \leq r-1$. Since $X$ is non-empty, we have by Lemma~\ref{lem:  intersections, if large, are non-empty} that $\vert X\vert \geq 1+ \delta^+(G_2\cap G_3)> 1+\frac{r-1}{r(r+1)}N$. Thus $\vert X\vert \leq r-1$ implies $N< \frac{(r-2)r(r+1)}{r-1}< r(r+1)$. Observe that for such values of $N$, $\lceil (1-\frac{1}{r})N\rceil+\lceil\frac{N}{r+1}\rceil \geq N$. 

Consider any edge $xx'\in G_3$. Since $\mathbf{G}$ contains no rainbow triangle, $N_{G_1}(x)$ must be disjoint from $N_{G_2}(x')$. As $\delta(G_1)+\delta(G_2)$ is at least $N$ by the observation above, this implies by the pigeon-hole principle that  $xx' \in G_1 \cap G_2 \cap G_3$. In particular, it follows that $G_3 \subseteq G_2$, and hence every vertex must lie in $X$. Thus $\vert X\vert =N$. Since $\vert X\vert\leq r-1$, this implies that $N-1\geq \delta(G_1)>(1-\frac{1}{r})N\geq N-1$, a contradiction.
\end{proof}
\begin{claim}\label{claim: G1 is Kr-free}
$G_1[X]$ is $K_r$-free.	
\end{claim}
\begin{proof}
Suppose for a contradiction that$\{x_i: \ i \in [r]\}$ is a set of $r$ distinct vertices from $X$ inducing a copy of $K_r$ in $G_1[X]$. By the minimum degree condition for colour $1$, $\sum_{i=1}^r \left(d_{G_1}(x_i)-(r-1)\right)> (r-1)(N-r)$, whence by the pigeon-hole principle there exists a vertex $u\in V=X\cup Y$ such that $\{x_i: \ i \in [r]\}\cup \{u\}$ induces a copy of $K_{r+1}$ in $G_1$.

Now, by the minimum degree condition for colours $2$ and $3$, we have 
\[d_{G_2}(u)+d_{G_3}(u)+\sum_{i=1}^r\left(d_{G_2}(x_i)+d_{G_3}(x_i)\right)> 2N,\] whence there exists $v$ sending at least three edges in colour $2$ or $3$ into $\{x_i: \ i\in [r]\}\cup\{u\}$. Since $\mathbf{G}$ contains no rainbow triangle, it follows that these three or more edges must all be in the same colour $c$. In particular, we may assume without loss of generality that $x_1v$ and $x_2v$ are both in $G_c$. 

Since $\mathbf{G}$ is Gallai, the sets $N_{G_2\cap G_3}(x_1)$, $N_{G_2\cap G_3}(x_2)$ and $N_{G_1}(v)$ are pairwise disjoint subsets of the $N$-set $V$. Further, by the definition of $X$ and by Lemma~\ref{lem:  intersections, if large, are non-empty}, the first two of these sets are non-empty and have size at least $\frac{r-1}{r(r+1)}N$, while our minimum degree assumption guarantees that the last set has size at least  $\delta(G_1)>\frac{r-1}{r}N$. Now for $r \geq 3$, we have
\[\frac{2(r-1)}{r(r+1)}N+\delta(G_1) >N,\]
contradicting the fact our three sets are pairwise disjoint subsets of $V$. 
\end{proof}
Now for every $y\in Y$, there exists by the construction of $Y$ a vertex $x\in X$ with $xy\in G_j$ for some colour $j\in \{2,3\}$. Since $\mathbf{G}$ contains no rainbow triangle, the sets $N_{G_1}(y)$ and $N_{G_2\cap G_3}(x)$ (the latter of which is a subset of $X$ by construction) are disjoint. In particular, by Lemma~\ref{lem:  intersections, if large, are non-empty} we have
\[\vert N_{G_1}(y)\cap X\vert \leq \vert X\vert -\vert N_{G_2\cap G_3}(x)\vert <\vert X\vert -\frac{r-1}{r(r+1)}N.\] 
It follows that $\vert G_1[X,Y]\vert <(N-\vert X\vert )\left( \vert X\vert -\frac{r-1}{r(r+1)}N\right)$. Further, since $\vert X\vert \geq r$ by Claim~\ref{claim: X is not too small}, we can apply Tur\'an's theorem~\cite{Turan41} to the $K_r$-free graph $G_1[X]$ to obtain the upper bound $\vert G_1[X]\vert < \frac{r-2}{r-1}\frac{\vert X\vert ^2}{2}$. However, by the minimum degree assumption for colour $1$,
\begin{align*} 
\frac{r-1}{r}N\vert X\vert <\delta(G_1)\vert X\vert \leq 2\vert G_1[X]\vert +\vert G_1[X,Y]\vert <\frac{r-2}{r-1}\vert X\vert ^2 +(N-\vert X\vert )\left( \vert X\vert -\frac{r-1}{r(r+1)}N\right).
\end{align*}
Setting $\vert X\vert =\alpha N$ and rearranging, we get 
 $\alpha^2 -\frac{2(r-1)}{r+1}\alpha +\frac{(r-1)^2}{r(r+1)}<0$, which has no solution when $r\geq 1$, so this last inequality yields the desired contradiction. 
\end{proof}
\begin{corollary}\label{cor: mindeg, no rainbow}
	There are no rainbow edges in $\mathbf{G}$.\qedhere
\end{corollary}


	

\noindent With Lemma~\ref{lem: mindeg, no 23 edge} in hand, we now show that bi-chromatic edges in colours $1j$ are widespread.
\begin{lemma}\label{lem: either every bichromatic degree is zero or large}
Let $j\in \{2,3\}$. Then either every vertex in $V$ is incident with a bi-chromatic edge in colours $1j$ or none is.
\end{lemma}
\begin{proof}
Assume without loss of generality $j=2$, and set $X=\{v: \ d_{G_1\cap G_2}(v)>0\}$ and $Y=V\setminus X$. Our goal is to show that $X\sqcup Y$ is a $1$-cut, which by Lemma~\ref{lem:1-cut} would imply that either every vertex is incident with an edge from $G_1\cap G_2$ (if $Y=\emptyset$) or none is (if $X=\emptyset$). Note that if $X\neq \emptyset$, then by Lemma~\ref{lem:  intersections, if large, are non-empty}, $\delta^+(G_1\cap G_2)>\frac{r-1}{r(r+1)}N$.
\begin{claim}
	$G_3[X,Y]=\emptyset$
\end{claim}
\begin{proof}
Suppose there exist $x\in X$ and $y\in Y$ are such that $xy\in G_3$.  Since $\mathbf{G}$ contains no rainbow triangle, we have that $N_{G_1\cap G_2}(x)$ is disjoint from $N_{G_1}(y)\cup N_{G_2}(y)$. Further, since $y\in Y$, $d_{G_1\cap G_2}(y)=0$, and $N_{G_1}(y)$ and $N_{G_2}(y)$ must be pairwise disjoint. This then implies,
\[N\geq d_{G_1}(y)+ d_{G_2}(y)+d_{G_1\cap G_2}(x)\geq \delta(G_1)+\delta(G_2)+ \delta^+(G_1\cap G_2) > \left(\frac{r-1}{r} + \frac{2(r-1)}{r(r+1)}\right)N>N,\]
a contradiction.
\end{proof}

\begin{claim}
	$G_2[X,Y]=\emptyset$
\end{claim}
\begin{proof}	
Recall that by Lemma~\ref{lem:  intersections, if large, are non-empty}, every vertex in $X$ is incident with at least $\frac{r-1}{r(r+1)}N$ bi-chromatic edges in colours $12$. Now suppose there exist $x\in X$ and $y\in Y$ with $xy\in G_2$. Since by the previous claim we have $G_3[X,Y]=\emptyset$, there exists $x'\in X$ with $xx'\in G_3$. Now $N_{G_1\cap G_2}(x')\subseteq X$, and hence is disjoint from $N_{G_3}(y)\subseteq Y$. Further, $N_{G_1}(x)$ is disjoint from both of these sets (as otherwise we have a rainbow triangle). However,
\[d_{G_1}(x)+d_{G_3}(y)+d_{G_1\cap G_2}(x') \geq \delta(G_1)+\delta(G_3)+\delta^+(G_1\cap G_2)>\left(\frac{r-1}{r} + \frac{2(r-1)}{r(r+1)}\right)N>N,\]
which yields a contradiction.
\end{proof}
\noindent
Thus there are no edges in colour $2$ or $3$ between $X$ and $Y$, and $X\sqcup Y$ is a $1$-cut as required.  Applying Lemma~\ref{lem:1-cut} we have that $X=\emptyset$ or $Y=\emptyset$, proving the lemma.
\end{proof}
\noindent

Since $\delta(G_1)+\delta(G_2)+\delta(G_3)>N$, it follows from the pigeon-hole principle that every vertex in $V$ is incident with at least one bi-chromatic edge, which by Lemma~\ref{lem: mindeg, no 23 edge} must be in colours $12$ or $13$.  Thus at least one of $G_1\cap G_2$, $G_1\cap G_3$ is non-empty. 

In the remainder of the proof, we shall assume without loss of generality that $G_1\cap G_3$ is non-empty. In particular, it follows from Lemmas~\ref{lem: either every bichromatic degree is zero or large} and~\ref{lem:  intersections, if large, are non-empty} that
 \begin{align*}
\delta(G_1\cap G_3)=\delta^+(G_1\cap G_3)>\frac{(r-1)N}{r(r+1)}.
 \end{align*}
\noindent 
We are now ready to prove our second key structural lemma and show that $G_2$ is bipartite. 
\begin{lemma}\label{lem: mindeg, G2 bipartite}
	The graph $G_2$ is bipartite.
\end{lemma}
\begin{proof}
We first show $G_2$ is triangle-free and that the neighbourhoods in $G_3$ are independent sets in $G_2$.
\begin{claim}\label{claim: no 2 triangle incident with a 13 edge}
The graph $G_2$ is triangle-free.
\end{claim}
\begin{proof}
Suppose $\{v_1, v_2,v_3\}$ induces a triangle in $G_2$. 
%
%
As $\mathbf{G}$ is a Gallai colouring template, $N_{G_1\cap G_3}(v_1)$, $N_{G_1}(v_2)$ and $N_{G_3}(v_3)$ must be pairwise disjoint subsets of $N$. 
Since  $\delta(G_1)+\delta(G_3)+\delta(G_1\cap G_3)>\left(\frac{r-1}{r} + \frac{2(r-1)}{r(r+1)}\right)N>N$, this is
	a contradiction.
\end{proof}
\begin{claim}\label{claim: G3 nhoods indep in G2}
	For every $v\in V$, $N_{G_3}(v)$ is an independent set in $G_2$.	
\end{claim}
\begin{proof}
Suppose $u,u'\in N_{G_3}(v)$ and $uu'\in G_2$. By Claim~\ref{claim: no 2 triangle incident with a 13 edge}, $N_{G_2}(u)$ and $N_{G_2}(u')$ are disjoint sets. Further, since there are no rainbow triangles, both of these sets are disjoint from $N_{G_1}(v)$. Since $\delta(G_1)+2\delta(G_2)>N$, this is a contradiction.
\end{proof}

\noindent
Suppose now that $G_2$ is not bipartite. Since $G_2$ is triangle-free by Claim~\ref{claim: no 2 triangle incident with a 13 edge}, it must contain a shortest odd cycle $C$ of length $\vert C\vert=(2\ell +1)$, for some integer $\ell \geq 2$. Let $v_1, \dots, v_{2\ell+1}$ denote the vertices of $C$. By the minimality of $C$, every vertex $v\in V$ (including the ones in $C$) can send at most two edges in $G_2$ to $C$, for otherwise we would have a shorter odd cycle. Furthermore, if a vertex $v \in V$ sends two edges of $G_2$ into $C$, these two vertices must lie at distance exactly $2$ apart on the cycle. Hence $N_{G_2}(v_i)$ is disjoint from $N_{G_2}(v_j)$ for any $j \not \in \left\{i-2,i,i+2\right\}$.

Fix an index $i$ and consider the sets $N_{G_2}(v_{i-2}) \cap N_{G_2}(v_i)$ and $N_{G_2}(v_{i+2}) \cap N_{G_2}(v_i)$. Since they are disjoint subsets of $N_{G_2}(v_i)$, at least one of these two sets contains fewer than $\frac{1}{2} \vert N_{G_2}(v_i) \vert$ elements, and so we may assume without loss of generality that $\vert N_{G_2}(v_{i-2}) \cap N_{G_2}(v_i)\vert \leq \frac{1}{2} \vert N_{G_2}(v_i) \vert$. Hence 
\begin{align} \label{Large union}
\vert N_{G_2}(v_{i-2}) \cup N_{G_2}(v_i) \vert \geq \frac{3}{2}\vert N_{G_2}(v_i) \vert \geq \frac{3\delta(G_2)}{2}
\end{align}

Now, $N_{G_3}(v_{i}) \cup N_{G_3}(v_{i-2})$ and $N_{G_1}(v_{i-1})$ are disjoint sets (else we have a rainbow triangle), and since $\delta(G_1)+2\delta(G_3)>N$ (by~\eqref{eq: ub and lb bound on delta(G_i)}), it follows that the sets $N_{G_3}(v_{i})$ and $N_{G_3}(v_{i-2})$ must have a non-empty intersection. Thus there exists $w$ for which we have $wv_{i-2}$, $wv_{i} \in G_{3}$. Since we have no rainbow triangles in $\mathbf{G}$, $N_{G_1}(w)$ is disjoint from $N_{G_2}(v_{i-2})\cup N_{G_2}(v_i)$; combining this information with~\eqref{Large union}, we deduce that
\begin{align*}
\frac{3}{2}\delta_2(G) \leq \vert N_{G_2}(v_{i-2}) \cup N_{G_2}(v_i) \vert  \leq N- \vert N_{G_1}(w)\vert \leq N-\delta(G_1)<\frac{N}{r},
\end{align*}
whence $\delta(G_2) < \frac{2N}{3r}$. Our next aim is to show that the same bound holds for $\delta(G_3)$ as well. To do this, in the next two claims we shall bound  the number of edges in the various colours a vertex $v\in V$ can send into $C$.
\begin{claim}\label{claim: not too many edges into cycle, r geq 3}
	For any $v \in V$, there are at most $2$ edges of $G_{3}$ between $v$ and $C$.
\end{claim}
\begin{proof}
Suppose for a contradiction that there exist $v \in V$ and $v_{i_1}, v_{i_2}, v_{i_3} \in C$ so that $vv_{i_j} \in G_{3}$ for every $j \in \left\{1,2,3\right\}$. Since $C$ contains  $2\ell +1\geq 5$ vertices, there exists at least one pair of indices from $\{i_1, i_2, i_3\}$ which do not differ by $2$ modulo $2\ell +1$, say $i_{1}$ and $i_{2}$. Hence $N_{G_2}(v_{i_1})$ and $N_{G_2}(v_{i_2})$ are disjoint sets, and they are both disjoint from $N_{G_1}(v)$. This contradicts the fact that $\delta(G_1) + 2\delta(G_2) > N$ (by~\eqref{eq: ub and lb bound on delta(G_i)}).	
\end{proof}
\begin{claim}\label{claim: vertices send few edges to C}
For every $v\in V$, $\vert G_1[\{v\}, C]\vert + \vert G_3[\{v\}, C]\vert\leq \vert C\vert $. Moreover this can be attained if and only if $v\notin C$ and $v$ sends no edge in colour $3$ to $C$.
\end{claim}
\begin{proof}

Observe that $v$ can send in total at most two edges in colours $1$ or $3$ to a neighbouring pair of vertices $\left\{v_{i-1},v_{i}\right\}$ on $C$, as otherwise we have a rainbow triangle. 

If $v$ sends an edge in $G_{1} \cap G_{3}$ to $v_{i}$ for some $i$, then $v$ does not send any edge in colours $1$ or $3$  to $v_{i-1}$ or $v_{i+1}$. By the averaging argument above, $v$ sends at most a total of $\vert C \vert - 3$ edges in colours $1$ or $3$ to the remaining $\vert C \vert - 3$ vertices, as required. 

Otherwise, every pair $vv_{i}$ can belong to at most one of $G_{1}$ or $G_{3}$. To attain the upper bound of $\vert C\vert$ on the total number of edges in colours $1$ or $3$ sent by $v$ into $C$, each such pair must be present in at least one of the graphs $G_1$ and $G_3$, which implies that all $\vert C \vert$ edges must be of the same colour (as otherwise we have a rainbow triangle). However, $v$ cannot send an edge of $G_3$ to every vertex on $C$, as neighbourhoods in $G_3$ are independent sets in $G_2$ (by Claim~\ref{claim: G3 nhoods indep in G2}). This completes the proof.

\end{proof}

\noindent Now let  $A := \left\{v \in V: \vert G_{3}[\{v\}, C] \vert > 0 \right\}$. Claim~\ref{claim: vertices send few edges to C} implies that we have 
\begin{align}\label{eq: A and C} 
\left(\delta(G_1) + \delta(G_3)\right) \vert C \vert \leq \vert C \vert \left(N - \vert A \vert \right) + \left(\vert C \vert - 1\right) \vert A \vert= \vert C\vert N-\vert A\vert.
\end{align}
Since every $v \in V$ sends most two edges of $G_3$ to $C$ (by Claim~\ref{claim: not too many edges into cycle, r geq 3}), we have by averaging over the vertices of $C$ that $\vert A \vert \geq \frac{\vert C \vert \delta(G_3)}{2}$. Combining this lower bound on $\vert A\vert$ with~\eqref{eq: A and C}, we obtain that $\delta(G_1) + \frac{3\delta(G_3)}{2} \leq N$, which implies that $\delta(G_3) < \frac{2N}{3r}$ and hence $\delta(G_2)+\delta(G_3)<\frac{4N}{3r}$. However, since $r \geq 3$, this contradicts the fact that $\delta(G_2) + \delta(G_3) > \frac{2N}{r+1}$.  From this final contradiction we deduce that $G_2$ contains no odd cycle, and hence that $G_2$ is bipartite as desired.
\end{proof}
\noindent Let $A\sqcup B$ be a bipartition of $V$ such that $G_2\subseteq (A, B)^{(2)}$. We now prove our third key structural lemma and show $G_3$ is well-behaved with respect to this bipartition:  either $G_3$ is a bipartite graph with the same bipartition, or it has no edges going from $A$ to $B$.
\begin{lemma}\label{lemma: G3 structure}
Either $G_3\subseteq (A, B)^{(2)}$ or $G_3\subseteq A^{(2)}\cup B^{(2)}$.
\end{lemma}
\begin{proof} \noindent We first show neighbourhoods in $G_3$ are well-behaved with respect to the bipartition $A\sqcup B$. 
\begin{claim}\label{claim: g3 respects bip}
For every $v\in V$, $N_{G_3}(v)$ has non-empty intersection with at most one of $A$ and $B$.
\end{claim}
\begin{proof}	
Suppose for a contradiction that $av, bv\in G_3$ for some $a\in A$, $b\in B$. Since $G_2$ is bipartite, $N_{G_2}(a)$ and $N_{G_2}(b)$ are disjoint. Furthermore, both of these sets are disjoint from $N_{G_1}(v)$ (as otherwise we would have a rainbow triangle). Since $\delta(G_1)+2\delta(G_2)>N$, this is a contradiction.
\end{proof}
\noindent With this claim in hand, we may define a new bipartition $V=C\sqcup D$ by setting $C:=\{v\in A: \ N_{G_3}(v)\subseteq B\}\cup \{v\in B:\ N_{G_3}(v)\subseteq A\}$ (i.e.\ the vertices incident with an edge of $G_3$ crossing the bipartition $A\sqcup B$) and $D:=\{v\in A: \ N_{G_3}(v)\subseteq A\}\cup \{v\in B:\ N_{G_3}(v)\subseteq B\}$ (i.e.\ the vertices incident with an edge of $G_3$ contained in $A^{(2)}\cup B^{(2)}$).  
\begin{claim}
$C\sqcup D$ is a $1$-cut.	
\end{claim}
\begin{proof}
By Claim~\ref{claim: g3 respects bip} we already know that there are no edges of $G_3$ in $(C, D)^{(2)}$. Suppose now $u\in C$, $v\in D$ and $uv\in G_2$. Since $G_2$ is bipartite, we may assume without loss of generality that $u\in A$ and $v\in B$. Let $w$ be a neighbour of $v$ in $G_3$, which must lie inside $B\cap D$ (since $v\in B\cap D$). Now as $\mathbf{G}$ contains no rainbow triangle, $N_{G_1}(v)$ must be disjoint from $N_{G_3}(u)$ and $N_{G_2}(w)$. Further $N_{G_2}(w)\subseteq A$ (since $G_2$ is bipartite and $w\in B$) and $N_{G_3}(u)\subseteq B$ (since $u\in A\cap C$). Thus the three sets $N_{G_1}(v)$, $N_{G_2}(w)$ and $N_{G_3}(u)$ are pairwise disjoint --- however $\delta(G_1)+\delta(G_2)+\delta(G_3)>N$, so this gives a contradiction. It follows there are no edges in $G_2$ from $C$ to $D$, and that $C\sqcup D$ is a $1$-cut as claimed. 	
\end{proof}
\noindent The Lemma now follows immediately from Lemma~\ref{lem:1-cut}: either $C=\emptyset$ or $D=\emptyset$.
%
%
\end{proof}
\noindent As a corollary to Lemma~\ref{lemma: G3 structure}, we can rule out the case $r=3$; knowing $r\geq 4$ will simplify the final analysis.
\begin{corollary}\label{cor: mindeg r >3}
	We must have $r\geq 4$.
\end{corollary}
\begin{proof}
	Suppose for a contradiction that $r= 3$. We may assume without loss of generality that $\vert A\vert \leq N/2\leq \vert B\vert$. We separate into two cases.

	If $G_3 \subseteq (A,B)^{(2)}$, then $\delta(G_1)>N/2$ implies that there is at least one edge $uv\in G_1$ with $u,v\in B$. Now both $N_{G_2}(u)$ and $N_{G_3}(v)$ are subsets of $A$ and the sum of their sizes is strictly greater than $N/2$. Hence by the pigeon-hole principle, $\vert N_{G_2}(u)\cap N_{G_3}(v)\vert>0$ and we have a rainbow triangle, a contradiction.

	On the other hand if $G_3\subseteq A^{(2)}\cup B^{(2)}$, then $\delta(G_1)>N/2$ implies that there is at least one edge $uv\in G_1$ with $u\in A$ and $v\in B$. Now both $N_{G_3}(u)$ and $N_{G_2}(v)$ are subsets of $A$ and the sum of their sizes is strictly greater than $N/2$. Again, by the pigeon-hole principle, this implies the existence of a rainbow triangle, a contradiction.	
\end{proof}
\noindent It turns out that analysing the two alternatives for $G_3$'s structure provided by Lemma~\ref{lemma: G3 structure} can be essentially done simultaneously if $G_{2}$ happened to be disconnected. This follows from the fact that in such a case, the only way to avoid a non-trivial $1$-cut is for $G_{3}$ to also be bipartite, with $G_{3}$ edges joining vertices in different connected components of $G_{2}$ (this is proved in Lemma~\ref{lemma: Four parts} below). In particular, the cases $G_{3} \subseteq A^{(2)} \cup B^{(2)}$ and $G_{3} \subseteq \left(A, B\right)^{(2)}$ are symmetric under switching the vertex classes of one connected component of $G_{2}$ in the bipartition.

Hence our next task is to prove that $G_{2}$ cannot be connected. Our approach is as follows: we show that if $x$ and $y$ are two vertices on the same side of the bipartition which are also in the same connected component of $G_2$, then they have many joint neighbours in $G_{2}$  (Lemma~\ref{lemma: G2 properties if connected} below). This in turn can be shown (Lemma~\ref{lemma: mindeg done if G2 connected}) to contradict the upper bound on $\Delta(G_{2})$ from~\eqref{eq: ub and lb bound on delta(G_i)}. Unfortunately, this approach does not quite work when $r = 4$ and $\delta(G_2)$ is small, and we need to apply a more ad hoc argument (Lemma~\ref{lemma: G2 not connected, r=4}) to deal with this special case.

\noindent Having sketched how our final analysis will go, we now make it concrete. We begin by recording a useful lemma which shows neighbourhoods in $G_2$ and $G_3$ have large intersections if they are non-empty.
\begin{lemma}\label{lemma: useful bound on nhood of pairs}
Let $\{i_1,j_1\}$ and $\{i_2, j_2\}$ be two copies of $\{2,3\}$. Let $u,u'$ be two distinct vertices in $V$. If $N_{G_{i_1}}(u)\cap N_{G_{i_2}}(u')\neq \emptyset$, then both $N_{G_{i_1}}(u)\cap N_{G_{i_2}}(u')$ and $N_{G_{j_1}}(u)\cap N_{G_{j_2}}(u')$ have size strictly larger than $\frac{(r-3)}{r(r+1)}N$. Furthermore, if $\max\left(\delta(G_{j_{1}}),\delta(G_{j_{2}})\right)=\max\left(\delta(G_2),\delta(G_3)\right)$, then we have the stronger bound $\vert N_{G_{j_1}}(u)\cap N_{G_{j_2}}(u') \vert > \frac{(r-1)}{r(r+1)}N$. 
\end{lemma}
\begin{proof}
Let $v\in N_{G_{i_1}}(u)\cap N_{G_{i_2}}(u')$. Since $\mathbf{G}$ is a Gallai colouring template, $N_{G_1}(v)$ is disjoint from $N_{G_{j_1}}(u) \cup N_{G_{j_2}}(u')$. Then by our minimum degree assumptions for our three colour classes and the lower bound on $\delta(G_c)$, $c\in \{2,3\}$ from~\eqref{eq: ub and lb bound on delta(G_i)}, we have:
\begin{align*}
\vert N_{G_{j_1}}(u)\cap N_{G_{j_2}}(u')\vert &\geq \vert N_{G_{j_1}}(u)\vert + \vert N_{G_{j_2}}(u')\vert - \left(N- \vert N_{G_1}(v)\vert \right)\\ &\geq \delta(G_1)+ \delta(G_{j_1})+\delta(G_{j_2}) -N> \frac{(r-3)}{r(r+1)}N,
\end{align*}
as required. Thus the intersection is non-empty, and the lower bound on $\vert N_{G_{i_1}}(u)\cap N_{G_{i_2}}(u')\vert$ follows symmetrically.	
Further, if at least one of $\delta(G_{j_{1}})$ or $\delta(G_{j_{2}})$ is equal to $\max\left(\delta(G_2),\delta(G_3)\right)$, then $\delta(G_{j_1})+\delta(G_{j_2})$ is at least $\frac{2N}{r+1}$, which, substituted in the displayed inequality above, yields an improved lower bound 
\begin{align*}
\vert N_{G_{j_1}}(u)\cap N_{G_{j_2}}(u')\vert > \frac{(r-1)}{r(r+1)}N.
\end{align*}
\end{proof}
The next lemma is the key tool we employ to derive a contradiction from our structural information; it guarantees that vertices from $A$ in the same component of $G_2$ have a large number of common neighbours in $G_2$. This in turn implies the existence of vertices in $B$ with large degree in $G_2$, violating the upper bound $\Delta(G_2)\leq \frac{N}{r}$ from~\eqref{eq: ub and lb bound on delta(G_i)}. 
\begin{lemma}\label{lemma: G2 properties if connected} 
Suppose that $\delta(G_2) \geq \frac{2}{3r}N$. If $u, u''$ are both in $A$ or both in $B$ and in addition lie in the same connected component of $G_2$, then they have at least $\frac{r-3}{r(r+1)}N$ neighbours in common in both $G_2$ and $G_3$. 
\end{lemma}
\begin{proof}
	Without loss of generality, we may assume $u,u''\in A$. Note it is enough to show that if there is a path of length $4$ in $G_2$ with edges $uv$, $vu'$, $u'v'$ and $v'u''$ joining $u$ to $v$, then $u$ and $u''$ have at least one common neighbour in $G_2$ --- indeed by Lemma~\ref{lemma: useful bound on nhood of pairs} it then follows that $u$ and $u''$ have at least $\frac{r-3}{r(r+1)}N$ common neighbours in both $G_2$ and $G_3$.

	Suppose for the sake of contradiction that $N_{G_2}(u)$ and $N_{G_2}(u'')$ are disjoint sets. Since $N_{G_2}(u) \cap N_{G_2}(u')$ and $N_{G_2}(u') \cap N_{G_2}(u'')$ are both non-empty, it follows from Lemma~\ref{lemma: useful bound on nhood of pairs} that $N_{G_3}(u) \cap N_{G_3}(u')$ and $N_{G_3}(u') \cap N_{G_3}(u'')$  are both non-empty as well. Now, for any $w\in N_{G_3}(u)\cap N_{G_3}(u')$, $N_{G_1}(w)$ is a set of size at least $\delta(G_1)>N-\frac{N}{r}$, which is disjoint from $N_{G_2}(u)\cup N_{G_2}(u')$; thus $\vert N_{G_2}(u)\cup N_{G_2}(u')\vert < \frac{N}{r}$, and similarly we have $\vert N_{G_2(u')}\cup N_{G_2}(u'')\vert<\frac{N}{r}$.

	Since $N_{G_2}(u)$ and $N_{G_2}(u'')$ are disjoint sets, at least one of the inequalities $\vert N_{G_2}(u) \cap N_{G_2}(u') \vert \leq \frac{1}{2}\vert N_{G_2}(u') \vert$ and $\vert N_{G_2}(u'') \cap N_{G_2}(u') \vert \leq \frac{1}{2}\vert N_{G_2}(u') \vert$ must be satisfied --- without loss of generality, let us assume it is the first of these two inequalities. Then it follows that 
	\begin{align*}
	\vert N_{G_2}(u) \cup N_{G_2}(u') \vert \geq \vert N_{G_2}(u) \vert  + \frac{1}{2} \vert N_{G_2}(u') \vert \geq \frac{3\delta(G_2)}{2}.
	\end{align*}
Since $\delta(G_2) \geq  \frac{2}{3r}N$, it follows that $	\vert N_{G_2}(u) \cup N_{G_2}(u') \vert \geq \frac{N}{r}$, a contradiction. The lemma follows. 
\end{proof}
\noindent We can now prove that $G_{2}$ is not connected. We split the proof into two lemmas, the first of which deals with the general case $r \geq 5$ and $r = 4$, $\delta(G_2)\geq \frac{2}{3r}N$, and the second of which deals with special case $r=4$, $\delta(G_2)< \frac{2}{3r}N$. 

\begin{lemma}\label{lemma: mindeg done if G2 connected}
Suppose $\delta(G_2)\geq \frac{2}{3r}N$. Then $G_{2}$ is not connected. 
\end{lemma}
\begin{proof}
Suppose for a contradiction that $G_2$ is connected. Assume without loss of generality that $\vert A \vert \leq \vert B \vert$.

\noindent \textbf{Case 1: $\delta(G_2) \geq \delta(G_3)$.} Lemma~\ref{lemma: G2 properties if connected} together with the `furthermore' part of Lemma~\ref{lemma: useful bound on nhood of pairs} implies that for every pair of vertices $b,b'\in B$, the joint neighbourhood $N_{G_2}(b)\cap N_{G_2}(b')$ is a subset of $A$ of size at least  $\frac{r-1}{r(r+1)}N$.  Combining this piece of information with the upper bound on $\Delta(G_2)$ from~\eqref{eq: ub and lb bound on delta(G_i)} and double counting, we get
\begin{align} 
  \frac{(r-1)N^3-2(r-1)N^2}{8r(r+1)} \leq \binom{\vert B\vert}{2} \frac{r-1}{r(r+1)}N &< \sum_{bb'\in B^{(2)}} \vert N_{G_2}(b)\cap N_{G_2}(b')\vert \notag\\
  &=\sum_{a\in A} \binom{d_{G_2}(a)}{2} \leq \vert A\vert \binom{\Delta_2(G)}{2}< \frac{N^3- rN^2}{4r^2}.\label{eq: case 1 L35 contradiction} 
\end{align}
Rearranging terms, this yields the inequality $\left(r(r-3)-2\right)N^3+4rN^2<0$, which is a contradiction since $r\geq 4$.

\noindent \textbf{Case 2: $\delta(G_3) > \delta(G_2)$.} By Lemma ~\ref{lemma: G3 structure} we have $G_{3} \subseteq \left(A,B\right)^{(2)}$ or $G_{3} \subseteq A^{(2)} \cup B^{(2)}$. If $G_{3} \subseteq \left(A,B\right)^{(2)}$, the proof follows by using exactly the same argument as in the previous case (with colours $2$ and $3$ interchanged), as Lemma~\ref{lemma: G2 properties if connected} implies that $G_{3}$ is also a connected bipartite graph with respect to the same vertex-partition as $G_{2}$. 

Hence it suffices to consider the case when $G_{3} \subseteq A^{(2)} \cup B^{(2)}$. 
By Lemma ~\ref{lemma: G2 properties if connected}, for every pair of vertices $b,b' \in B$ the joint neighbourhood $N_{G_2}(b)\cap N_{G_2}(b')$ is non-empty, and thus by the `furthermore' part of Lemma~\ref{lemma: useful bound on nhood of pairs} the joint neighbourhood $N_{G_3}(b)\cap N_{G_3}(b')$ is a subset of $B$ of size at least $\frac{r-1}{r(r+1)}N$. Hence in a similar fashion to the previous case, using the upper bound on $\Delta(G_3)$ from~\eqref{eq: ub and lb bound on delta(G_i)} and double counting, we have 
\begin{align*} 
  \left(\frac{\vert B\vert N}{2r}\right)\frac{\left(\vert B\vert -1 \right)(r-1) }{r+1}=\binom{\vert B\vert}{2} \frac{r-1}{r(r+1)}N &< \sum_{bb'\in B^{(2)}} \vert N_{G_3}(b)\cap N_{G_3}(b')\vert \\
  &=\sum_{b\in B} \binom{d_{G_3}(b)}{2} \leq \vert B\vert \binom{\Delta_3(G)}{2}<\left(\frac{\vert B\vert N}{2r} \right) \frac{N-r}{r}.
\end{align*}
Rearranging the inequality above and using $\vert B\vert \geq N/2$ then yields the desired contradiction:
\[  0> \vert B\vert (r-1)r - (r+1) N   +2r\geq   \frac{r(r-3)-2}{2}N + 2r>0.\]
\end{proof}

\begin{lemma}\label{lemma: G2 not connected, r=4}
Suppose that $r = 4$ and $\delta(G_2) < \frac{N}{6}=\frac{2}{3}\cdot\frac{N}{4}$. Then $G_{2}$ is not connected. 
\end{lemma}

\begin{proof}
If $G_{3}$ is also bipartite with respect to the vertex classes $A$ and $B$, we are done as in Case~1 in the proof of Lemma ~\ref{lemma: mindeg done if G2 connected} (with colours $2$ and $3$ interchanged): $G_2$ being connected implies $G_3$ is a bipartite connected graph with minimum degree $\delta(G_3)>\frac{2N}{5}-\delta(G_2)>\frac{N}{6}$, which leads to a contradiction as in~\eqref{eq: case 1 L35 contradiction}.

Hence we may assume that $G_{3} \subseteq A^{(3)} \cup B^{(3)}$.  Set $\vert A \vert = \alpha N$ and $\vert B \vert = \beta N$, and assume without loss of generality that $\alpha \leq \frac{1}{2} \leq \beta$. We claim that $\vert A\vert$ and $\vert B\vert$ have almost balanced sizes, i.e.\ that $\beta$ is not much larger than $1/2$. 
\begin{claim}\label{claim> beta is not too large, r=4 conn}
	 $\beta < \frac{15}{29}$
\end{claim}
\begin{proof}
Since $\delta(G_1) > \frac{3N}{4}$, by averaging there exists a vertex $b \in B$ sending at least
\begin{align*}
\frac{\vert G_1\left[A,B\right]\vert }{\vert B \vert} \geq \frac{\left(\delta(G_1)-\vert A\vert\right)\vert A\vert}{\vert B\vert}  >\frac{\left(\frac{3N}{4} - \alpha N\right) \alpha N}{\beta N} =\frac{\left(\beta-\frac{1}{4}\right) (1-\beta) }{\beta}N.
\end{align*}
edges of $G_1$ into $A$. Let $a\in N_{G_2}(b)$. Then 
$N_{G_3}(a)$ and $N_{G_1}(b)\cap A$ are disjoint subsets of $A$ (otherwise we have a rainbow triangle), and so
\begin{align*}
\vert A\vert = (1-\beta) N\geq \vert N_{G_3}(a)\vert + \vert N_{G_1}(b)\cap A\vert \geq\delta(G_3)+ \frac{ \left(\beta-\frac{1}{4}\right) (1-\beta) }{\beta}N.
\end{align*}
Since $\delta(G_3) > \frac{2N}{5} - \frac{N}{6} = \frac{7N}{30}$, solving the inequality above yields $\beta < \frac{15}{29}$ as desired. 
\end{proof}

\begin{claim}\label{claim: common neighbours in col 3 in B, r=4}
For any two vertices $a ,a' \in A$, there exist $b\in B$ and $a''\in A$ such that $aa''$, $a'a''$ both lie in $G_{3}$ and $ab$, $a'b$ both lie in $G_2$.
\end{claim}
\begin{proof}
Fix $a,a'\in A$. By Lemma~\ref{lemma: useful bound on nhood of pairs}, these two vertices have a joint neighbour in $G_2$ (which must lie in $B$) if and only if they have a joint neighbour in $G_3$ (which must lie in $A$).	Observe that by ~\eqref{eq: ub and lb bound on delta(G_i)} $\Delta(G_3)<\frac{N}{r}=\frac{N}{4}$.
	

  Suppose for a contradiction that $N_{G_2}(a)$, $N_{G_2}(a')$ are disjoint subsets of $B$. Since $\vert A  \vert \leq \frac{N}{2}$ and $\delta(G_1) > \frac{3N}{4}$, $\vert N_{G_1}(a)\cap N_{G_1}(a')\vert > \vert A\vert$, and hence there exists $b \in B$ such that  $ab$, $a'b$ both lie in $G_1$.  But then $N_{G_3}(b)$, $N_{G_2}(a)$ and $N_{G_{2}}(a')$ must be pairwise disjoint subsets of $B$ (otherwise we have a rainbow triangle). However, as $\Delta(G_3) \leq \frac{N}{4}$, we have $\delta(G_2) > \frac{2N}{5}-\frac{N}{4}=\frac{3N}{20}$; together with the bound $\vert B\vert\leq \frac{15N}{29}$ from Claim~\ref{claim> beta is not too large, r=4 conn}, this yields a contradiction:
\begin{align*}
\vert B\vert \geq \vert N_{G_{3}}(b)\vert + \vert N_{G_2}(a) \vert + \vert N_{G_2}(a') \vert \geq 2\delta(G_2) + \delta(G_3) > \frac{11N}{20} >\frac{15}{29N}\geq  \vert B \vert.
\end{align*}
\end{proof}
\noindent
Thus any two vertices $a,a' \in A$ have a common neighbour in $G_3[A]$, and hence strictly more than $\frac{3N}{20}$ such neighbours by Lemma~\ref{lemma: useful bound on nhood of pairs}. Arguing as in the proof of Lemma~\ref{lemma: mindeg done if G2 connected} we double count edges of $G_3[A]$ and use the upper bound on $\Delta(G_3)$ from~\eqref{eq: ub and lb bound on delta(G_i)} to get 
\begin{align*}
\binom{\vert A\vert}{2} \frac{3N}{20} < \vert A\vert \binom{\Delta(G_3)}{2} \leq \frac{\vert A \vert}{2} \frac{N}{4}\left(\frac{N}{4}-1\right).
\end{align*}
Dividing through by $\vert A\vert N/40 $ and rearranging terms, we get $\vert A\vert < \frac{5}{12}N-\frac{2}{3}< \frac{14N}{29}$, contradicting Claim~\ref{claim> beta is not too large, r=4 conn}.
It follows that $G_2$ cannot be connected.
\end{proof}

\noindent By~\eqref{eq: ub and lb bound on delta(G_i)}, for $r\geq 5$ we have $\delta(G_2)\geq \frac{r-1}{r(r+1)}\geq\frac{2}{3r}$. Thus Lemma~\ref{lemma: mindeg done if G2 connected} and Lemma~\ref{lemma: G2 not connected, r=4} together imply that $G_2$ is disconnected. In particular, there exist partitions $A = V_{1} \sqcup V_{4}$ and $B = V_{2} \sqcup V_{3}$ of $A$ and $B$ into a total of four disjoint non-empty sets $V_{1},\dots,V_{4}$ so that there are no edges of colour $2$ in $(V_{1}\cup V_{2}, V_{3}\cup V_{4})^{(2)}$. Furthermore, we may assume without loss of generality that $G_{2}\left[V_{1} \cup V_{2}\right]$ is connected.

 Our next aim is to show that $G_{3}$ is also bipartite with edges only in $\left(V_{1},V_{4}\right)^{(2)}$ and $\left(V_{2},V_{3}\right)^{(2)}$ or only in $\left(V_{1},V_{3}\right)^{(2)}$ and $\left(V_{2},V_{4}\right)^{(2)}$. Of course, these two cases are identical up to interchanging the labels of $V_3$ and $V_4$, which will allow us to treat the two alternatives $G_3 \subseteq \left(A,B\right)^{(2)}$ and $G_3 \subseteq A^{(2)} \cup B^{(2)}$ from Lemma~\ref{lemma: G3 structure} simultaneously. 

\begin{lemma} ~\label{lemma: Four parts}
There exists a choice $\left\{i,j\right\} = \left\{3,4\right\}$ so that $G_{3} \subseteq \left(V_{1},V_{i}\right)^{(2)} \cup \left(V_{2},V_{j}\right)^{(2)}$. Furthermore, each of the graphs $G_{2}\left[V_{1},V_{2}\right]$, $G_{2}\left[V_{3},V_{4}\right]$, $G_{3}\left[V_{1},V_{i}\right]$ and $G_{3}\left[V_{2},V_{j}\right]$ is connected. 
\end{lemma}
\begin{proof} We first show neighbourhoods in colours $2$ and $3$ are well-behaved with respect to the partition $V=\sqcup_{i=1}^4 V_i$.
\begin{claim}\label{claim: vertices send 2,3 edges to at most one Vi}
For every vertex $v\in V$ and every colour $c\in \{2,3\}$, $v$ sends edges in colour $c$ to exactly one of the four parts $V_i$, $i\in [4]$.
\end{claim}
\begin{proof}
Our assumptions on $G_2$ ensure that there is no vertex sending an edge of $G_2$ to two distinct parts $V_i$, $i\in [4]$, and thus no pair of vertices in distinct parts have a common neighbour in  $G_2$. By Lemma~\ref{lemma: useful bound on nhood of pairs}, this implies that no pair of vertices in distinct parts have a common neighbour in $G_3$, and hence that every vertex sends edges in colour $3$ to exactly one of the four parts $V_i$, $i\in [4]$.
\end{proof}
\noindent Let $W_{1} = \bigcup_{v\in V_1} N_{G_{3}}(v)$ and $W_{2} = \bigcup_{v\in V_2}N_{G_{3}}(v)$. 
\begin{claim}\label{claim: w1 in Vi}
 There exist distinct $i,j$ with $\left\{i,j\right\} = \left\{1,2\right\}$ or $\left\{i,j\right\} = \left\{3,4\right\}$ so that $W_{1} \subseteq V_{i}$ and $W_{2} \subseteq V_{j}$.
\end{claim} 
\begin{proof}
Claim~\ref{claim: vertices send 2,3 edges to at most one Vi} implies that $W_1$ and $W_2$ are disjoint. Further, by Lemma~\ref{lemma: useful bound on nhood of pairs}, if two vertices have a common neighbour in $G_2$, then they have a common neighbour in $G_3$ and hence, by Claim~\ref{claim: vertices send 2,3 edges to at most one Vi} again,  they must send edges of $G_3$ to the same part.

Since $G_2[V_1, V_2]$ is connected, it follows that $W_1$ and $W_2$ are subsets of two distinct parts $V_i$ and $V_j$ respectively, $i,j\in [4]$.  Now picking any edges $v_1v_2\in G_2$ and $v_2w_2\in G_3$ with $v_1\in V_1$, $v_2\in V_2$ and $w_2\in W_2$, we see by yet another application of Lemma~\ref{lemma: useful bound on nhood of pairs} that there exists $w_1 \in N_{G_3}(v_1)\cap N_{G_2}(w_2)$. Since $w_1\in W_1\subseteq V_i$ and $w_2\in V_j$, $w_1w_2\in G_2$ implies that $\{i,j\}=\{1,2\}$ or $\{3,4\}$.
\end{proof}
\noindent Set $X = V_{1} \cup V_{2} \cup W_{1} \cup W_{2}$ and $Y = V \setminus X$. 
 \begin{claim}
 	 $X \sqcup Y$ is a $1$-cut.
 \end{claim}
\begin{proof}
By Claims~\ref{claim: vertices send 2,3 edges to at most one Vi} and~\ref{claim: w1 in Vi}, we have that for  every vertex $w_1\in W_1$, $N_{G_3}(w_1)\subseteq V_1$, and similarly $N_{G_3}(w_2)\subseteq V_2$ for every $w_2\in W_2$. Thus there are no edges in colour $3$ from $X$ to $Y$.

Let $w_1\in W_1$ and let $w_2$ be such that $w_1w_2\in G_2$. Then by construction there exists $v_1\in V_1$ such that $v_1w_1\in G_3$. Applying Lemma~\ref{lemma: useful bound on nhood of pairs}, we see that $N_{G_2}(v_1)\cap N_{G_3}(w_2)$ is non-empty, which implies $w_2$ send an edge in colour $3$ to some $v_2\in V_2$, and hence that $w_2\in W_2$. Thus all edges in colour $2$ from $W_1$ lie in $W_2$; a symmetric argument shows that all edges in colour $2$ from $W_2$ lie in $W_1$. By construction, we already know all edges in colour $2$ from $V_1\sqcup V_2$ are to $V_1\sqcup V_2$. Thus there are no edges in colour $2$ from $X$ to $Y$, and $X\sqcup Y$ is a $1$-cut as claimed.
%
\end{proof}
\noindent
If $W_{1}, W_2 \subseteq V_{1} \cup V_{2}$, then $X = V_{1}\cup V_{2}$ and $Y\neq \emptyset$, which in turn implies that the $1$-cut $V = X \sqcup Y$ is non-trivial, contradicting Lemma~\ref{lem:1-cut}.  Thus $W_1, W_2\subseteq V_3\cup V_4$. By symmetry we may assume that $W_{1} \subseteq V_{4}$, whence $W_{2} \subseteq V_{3}$. For this $1$-cut to be trivial, we must then have $Y=\emptyset$ and so $W_{1} = V_{4}$ and $W_{2} = V_{3}$.  
It then follows from Lemma~\ref{lemma: G3 structure} that $G_{3} \subseteq \left(V_{1},V_{4}\right)^{(2)} \cup  \left(V_{2},V_{3}\right)^{(2)}$.

To complete the proof of the lemma, we need to verify that the graphs $G_{2}$ and $G_{3}$ are connected when restricted to the appropriate vertex class pairs. Let $i \in \left\{1,2\right\}$, and consider a pair of vertices $u, v \in V_{i}$. Since $G_{2}\left[V_{1},V_{2}\right]$ is connected, there exists a path $u = u_{0},\dots,u_{2n} = v$ in $G_{2}$ between them, with $u_{2t} \in V_{i}$ for each $t$. Since $u_{2t}$ and $u_{2(t+1)}$ have a common $G_{2}$ neighbour, Lemma ~\ref{lemma: useful bound on nhood of pairs} implies that they also have a common $G_{3}$ neighbour $w_{2t+1}$ in $V_{5-i}$. Thus $G_{3}\left[V_{i},V_{5-i}\right]$ is connected for both $i \in \left\{1,2\right\}$, and a symmetrical argument proves that $G_{2}\left[V_{3},V_{4}\right]$ is also connected. 
\end{proof}	
\noindent We are at last in a position to obtain a contradiction, in a similar spirit to the proof of Lemma ~\ref{lemma: mindeg done if G2 connected}. We first show $r$ is large, and then double count edges to conclude the proof of Theorem~\ref{theorem: r >2 case}.
\begin{claim}\label{claim: mindeg r>7}
We must have $r\geq 7$.
\end{claim}
\begin{proof}
Note our minimum degree conditions imply each of the four parts 
$V_{1},\dots,V_{4}$ are non-empty and have size at least $\max\left(\delta(G_2), \delta(G_3)\right)>\frac{N}{r+1}$. Assume without loss of generality that $V_{1}$ is the largest of the four parts, so that in particular $\vert V_{1}\vert \geq \frac{N}{4}$. Since $\delta(G_1)> (1-\frac{1}{r})N$ and since $r\geq 4$, for any $u \in V_3$ there must be a vertex $v\in V_1$ such that $uv\in G_1$.

Now, $N_{G_2}(v)$ and $N_{G_3}(u)$ are disjoint subsets of $V_2$, whence $\vert V_2 \vert \geq  \delta(G_2)+\delta(G_3)>\frac{2N}{r+1}$. Similarly $\vert V_4\vert >\frac{2N}{r+1}$.  As $\vert V_1\vert \geq \vert V_2\vert$ and $\vert V_3\vert >\frac{N}{r+1}$, this implies 
\[ N=\vert V_1\vert + \vert V_2\vert +\vert V_3\vert +\vert V_4\vert >\frac{7N}{r+1},\]
whence $r\geq 7$, as claimed.
\end{proof} 
\noindent
We may assume that $\delta(G_2) \geq \delta(G_3)$, that $\vert V_1\vert +\vert V_2\vert = N'\geq N/2$ and that $\vert V_1\vert \geq \vert V_2\vert$.  By Lemma~\ref{lemma: G2 properties if connected}, for every pair of vertices $b, b'\in V_1$ we have that $N_{G_2}(b)\cap N_{G_2}(b')$ is a subset of $V_2$ of size at least $\frac{r-1}{r(r+1)}N$. On the other hand, we know from~\eqref{eq: ub and lb bound on delta(G_i)} that each vertex $a\in V_2$ receives at most $\Delta(G_2)<N/r$ edges in colour $2$ from $V_1$. By double-counting as in the proof of Lemma ~\ref{lemma: mindeg done if G2 connected}, we get
\begin{align*}
\frac{(r-1)(N')^2 N-2(r-1)N'N}{8r(r+1)}\leq \binom{\vert V_1\vert}{2}\frac{r-1}{r(r+1)}N\leq \vert V_2\vert \binom{\Delta (G_2)}{2}<\frac{N'N^2-r N'N}{4r^2}.
\end{align*}
Rearranging terms, this yields the inequality
\[  N' N \left(  N'(r^2-r) - N(2r+2) + 4r \right) <0.\]
Since $N'(r^2-r)-N(2r+2)\geq \frac{N}{2}\left(r^2-5r-4\right)>0$ for $r\geq 7$ (which holds by Claim~\ref{claim: mindeg r>7}), this is the final contradiction.
\end{proof}

\subsection{Proof of the case $r = 2$}\label{subsection: proof of r=2}
\begin{proof}[Proof of Theorem~\ref{theorem: r =2 case}]
We now turn our attention to the $r = 2$ case. We shall follow a line of argument similar to that we used in the $r \geq 3$ case; however, as we outlined in Section~\ref{subsection: proof strategy mindegree}, for $r=2$ the crucial inequality~\eqref{eq: key ineq r at least 3} fails in general, and we must make a number of technical detours from our main proof path in order to circumvent the resulting obstacles. In particular, unlike in the proof of Theorem~\ref{theorem: r >2 case}, we will not treat the colours $2$ and $3$ interchangeably in our argument.

Suppose $\mathbf{G}$ is a Gallai $3$-olouring template on $N$ vertices with $\delta(G_1)>\frac{N}{2}$, $\delta(G_2)\geq \delta(G_3)>0$ and $\delta(G_2)+\delta(G_3)>\frac{2N}{3}$. 
From~\eqref{eq: (r-1)N/(r(r+1)) bound on delta G3} and Proposition~\ref{prop: trivial min degree bound}, we obtain the trivial bounds $\frac{N}{3} < \delta(G_2) < \frac{N}{2}$ and $\delta(G_3) > \frac{N}{6}$. Further, we have by Lemma~\ref{lem:  intersections, if large, are non-empty} (which holds for $r=2$, as noted in Remark~\ref{remark: also works for r=3}) that $\delta^+(G_i\cap G_j)>\frac{N}{6}$ for all distinct colours $i,j$ for which there exists a bi-chromatic edge in colours $ij$. As in the previous section, given a colour $c\in \{2,3\}$ we denote by $\overline{c}$ the complementary colour, $\{\overline{c}\}:=\{2,3\}\setminus \{c\}$.

 We begin as in the proof of Theorem~\ref{theorem: r >2 case} by proving that $G_{2} \cap G_{3} = \emptyset$. 
\begin{lemma}\label{lem: mindeg, no 23 edge r=2}
$G_2\cap G_3 =\emptyset$.	
\end{lemma}
\begin{proof}
Again, set $X$ to be the set of vertices in $V$ incident with a bi-chromatic edge in colours $23$ from the  graph $G_2\cap G_3$. Suppose for a contradiction that $X\neq \emptyset$. Let $Y$ denote the set of $y\in V\setminus X$ such that $y$ sends at least one edge in colour $2$ or $3$ to $X$. Finally, let $Z:= V\setminus (X\cup Y)$ denote the rest of the vertices in $V$. We will analyse the structure of the edges between and inside these three sets in a succession of claims.
%
\begin{claim}\label{cl: no 1-edge in X}
$G_1[X]$ is empty. 
\end{claim}
\begin{proof}
Suppose there exists $x, x' \in X$ with $xx' \in G_{1}$. Since $\delta(G_1)>\frac{N}{2}$, there exists $v$ such that $xx'v$ induces a triangle in $G_{1}$.  Further, since $3\left(\delta(G_2) + \delta(G_3)\right) > 2N$, there exist $w \in V$ and $c \in \left\{2,3\right\}$ so that $\left\{x,x',v\right\} \subseteq N_{G_c}(w)$.  As we have no rainbow triangles in $\mathbf{G}$, the sets $N_{G_1}(w)$, $N_{G_2 \cap G_3}(x)$, $N_{G_2 \cap G_3}(x')$ and $N_{G_{\overline{c}}}(v)$ are pairwise disjoint (and non-empty, since $x,x'\in X$). Since \[\delta(G_1) + 2\delta^+(G_2\cap G_3)+\min(\delta(G_2),\delta(G_3)) > \frac{N}{2}+2\cdot\frac{N}{6}+\frac{N}{6}>N,\] this gives a contradiction. 
\end{proof}
\begin{claim}\label{cl: Z and XcupZ large}
	$\vert Z\vert  > \frac{N}{3}$.
\end{claim}
\begin{proof}
	For any $y \in Y$, there exists $x \in X$ so that $xy \in G_{2} \cup G_{3}$. Since $N_{G_2 \cap G_3}(x)$ contains at least $\delta^+(G_2\cap G_3)>\frac{N}{6}$ elements and is a subset of $X$ disjoint from $N_{G_1}(y)$, it follows that $\vert G_1[X,Y] \vert \leq \vert Y\vert \left(\vert X\vert - \frac{N}{6}\right)$. Since $X$ contains no edge of $G_1$ by Claim~\ref{cl: no 1-edge in X}, counting edges in colour $1$ from $X$ we deduce that 
	\begin{align*}
	\vert X\vert \frac{N}{2} < \vert X\vert \delta(G_1) < \left\vert G_1[X,Y\cup Z] \right\vert \leq \vert Y\vert \left(\vert X\vert - \frac{N}{6}\right) + \vert X\vert \cdot \vert Z\vert.
	\end{align*}
Using $\vert Y\vert =N-\vert X\vert -\vert Z\vert$ and rearranging terms, we obtain the bound
	\begin{align*}
	\vert Z\vert > \frac{\vert X\vert^2 - \frac{2N\vert X\vert}{3} + \frac{N^2}{6}}{N/6} = \frac{\left(\vert X\vert -\frac{N}{3}\right)^2}{N/6} +\frac{N}{3}.
	\end{align*}
The claim follows immediately.
\end{proof}
\begin{claim} \label{claim: 3's isolated}
	$G_{3} \subseteq X^{(2)} \sqcup Y^{(2)} \sqcup Z^{(2)}$.
\end{claim}
\begin{proof}
	By construction, there are no edges in colour $3$ from $X$ to $Z$. If there exist $y\in Y$, $z\in Z$ with $yz\in G_3$, then by construction there exist $x\in X$ with $xy\in G_c$ for some $c\in \{2,3\}$. Since we have no rainbow triangles, $N_{G_2\cap G_3}(x)$ (which by construction is a subset of $X$), $N_{G_1}(y)$ and $N_{G_2}(z)$ (which by construction is a subset of $Y \cup Z$) must be pairwise disjoint. However the sum of the sizes of these three sets is at least $\delta^+(G_2\cap G_3)+ \delta(G_1)+\delta(G_2)>N$, a contradiction. Thus there are no edges in colour $3$ from $X\cup Y$ to $Z$.

	Suppose now there exist $xy\in G_3$ with $x\in X$ and $y\in Y$. Since $\vert Z\vert > \frac{N}{3}$ (by Claim~\ref{cl: Z and XcupZ large}), since $N_{G_2}(y) \cap N_{G_{3}}(y) = \emptyset$ (by construction of $Y$) and since $N_{G_3}(y)\cap Z=\emptyset$ (as shown in the previous paragraph), it follows from our minimum degree assumption $\delta(G_2)+\delta(G_3)>\frac{2N}{3}$ that there exists $z \in Z$ such that $yz \in G_{2}$. Now, the sets $N_{G_2}(x)$ and $N_{G_3}(z)$ are disjoint, as the first one is contained in $X \cup Y$ while the second is contained in $Z$. Further, both these sets are disjoint from $N_{G_1}(y)$ (else we have a rainbow triangle). This contradicts the fact that $\delta(G_1) + \delta(G_2) + \delta(G_3) > N$, and the claim follows. 
\end{proof}

\begin{claim}\label{cl:no 2 edge in Y, no 1 edge in X,Y}
$G_2[Y]$ and $G_1[X,Y]$ are both empty.
\end{claim}
\begin{proof}
Note that for any $y \in Y$ there exist $x \in X$ and $z \in Z$ so that $xy, yz \in G_2$. Indeed, the existence of $x$ follows from the construction of $Y$ and Claim ~\ref{claim: 3's isolated}, while the existence of $z$ follows from the facts that $N_{G_2}(y)\cap N_{G_3}(y)=\emptyset$ and $\vert Z\vert > \frac{N}{3}>N - \left(\delta(G_2) + \delta(G_3)\right)$ (by Claim~\ref{cl: Z and XcupZ large}) together with Claim ~\ref{claim: 3's isolated}. Since $N_{G_{3}}(x) \subseteq X$ and $N_{G_{3}}(z) \subseteq Z$ have size at least $\delta(G_3)$, and are both disjoint from $N_{G_{1}}(y)$ (else we have a rainbow triangle), we must have that
\[ \vert N_{G_1}(y)\cap Y\vert \geq \vert N_{G_1}(y)\vert- \vert X\setminus N_{G_{3}}(x) \vert -\vert Z\setminus N_{G_3}(z)\vert \geq  \delta(G_1)+2\delta(G_3)-\left(N-\vert Y\vert\right)>  \vert Y\vert +2\delta(G_3)-\frac{N}{2}.\]
Since $y\in Y$ was arbitrarily chosen, this implies that $\delta(G_1[Y])>\vert Y\vert +2\delta(G_3)-\frac{N}{2}$, and hence by Claim~\ref{claim: 3's isolated} that
\begin{align*}
\delta(G_1[Y]) + \delta(G_3[Y]) \geq \vert Y \vert + 3\delta(G_3) - \frac{N}{2} > \vert Y \vert,
\end{align*}
which in turn implies that there cannot be any edge of $G_2$ inside $Y$ (by the pigeon-hole principle: any such edge would yield a rainbow triangle in $Y$). This proves the first part of our claim.

For the second part, suppose that there exist $y \in Y$, $x\in X$ and $z\in Z$ so that $xy, yz \in G_{1}$. Then $N_{G_3}(x)\subseteq X$, $N_{G_3}(z)\subseteq Z$, $N_{G_3}(y)\subseteq Y$ are pairwise disjoint subsets by Claim~\ref{claim: 3's isolated}. 
Further, all three sets are disjoint from $N_{G_2}(y)$ (else we have a rainbow triangle or an edge of $G_2$ inside $Y$, violating the first part of our claim). 
Since $3\delta(G_3)+\delta(G_2)>2\cdot\frac{N}{6}+\frac{2N}{3}=N$, this gives a contradiction. Hence a vertex $y \in Y$ can send edges of $G_1$ to at most one of $X$ and $Z$. 

By construction and Claim~\ref{claim: 3's isolated}, for any $y\in Y$ there exists $x\in X$ with $xy\in G_2$. By Claim~\ref{claim: 3's isolated} and the 
absence of rainbow triangles, $N_{G_1}(y)\cap X$ and $N_{G_3}(x)$ are disjoint subsets of $X$. Thus $\vert N_{G_1}(y)\cap X\vert $ (which is either at least $\delta(G_1)-\vert Y\vert $ or zero by the previous paragraph) is at most 
\[\vert X\vert -\delta(G_3)< N-\vert Y\vert -\vert Z\vert - \frac{N}{6}< \frac{N}{2}-\vert Y\vert<\delta(G_1)-\vert Y\vert,\]
where in the penultimate inequality we used the bound $\vert Z\vert>\frac{N}{3}$ from Claim~\ref{cl: Z and XcupZ large}. We deduce that $\vert N_{G_1}(y)\cap X\vert =0$, and hence that there are no edges in colour $1$ from $X$ to $Y$.
\end{proof}
\noindent Claims~\ref{cl: no 1-edge in X} and~\ref{cl:no 2 edge in Y, no 1 edge in X,Y} together imply that $N_{G_1}(x) \subseteq Z$ for any $x \in X$, and hence we must have $\vert Z\vert  \geq \delta(G_1)> \frac{N}{2}$.
\begin{claim}\label{cl: no 1-edge in Y}
$G_1[Y,Z]$ is empty.
\end{claim}
\begin{proof}
Suppose there exist $z \in Z$ and $y \in Y$ so that $yz \in G_{1}$. Since $\vert Z\vert > \frac{N}{2}>N-\delta(G_1)$, there exists $z' \in Z$ so that $zz' \in G_{1}$. Note that by Claim~\ref{claim: 3's isolated} and the absence or rainbow triangles, $N_{G_3}(y)\subseteq Y$, $N_{G_3}(z')\subseteq Z$, $X$ and $N_{G_2}(Z)\subseteq V\setminus X$ are disjoint sets. Since $\vert X \vert \geq \delta(G_3)$ (by Claim~\ref{claim: 3's isolated}), this gives a contradiction:
\begin{align*}
\vert N_{G_3}(y)\vert +\vert N_{G_3}(z')\vert +  \vert X\vert +\vert N_{G_2}(Z)\vert\geq   3\delta(G_3) + \delta(G_2) > 2\delta(G_3) + \frac{2N}{3} > N.
\end{align*}
\end{proof}
\noindent 
Since $\vert Z\vert >\frac{N}{2}$, as observed above Claim~\ref{cl: no 1-edge in Y}, Claim~\ref{cl: no 1-edge in Y} together with our minimum degree bound $\delta(G_1)>\frac{N}{2}$ implies $Y=\emptyset$. Since by construction there are no edges of colours $2$ or $3$ between $X$ and $Z$ and no bi-chromatic edges in colours $23$ in $Z$, we must have $\vert Z \vert \geq \delta(G_2)+\delta(G_3)>\frac{2N}{3}$ and $\vert X\vert \geq \delta(G_2)>\frac{N}{3}$. This at last is our contradiction. Thus $X=\emptyset $ and  $G_{2} \cap G_{3} = \emptyset$.  
\end{proof}
\noindent
Consider now a new partition of $V$ into $X = \left\{v: N_{G_1 \cap G_{2}}(v) \neq \emptyset\right\}$ and $Y = V \setminus X$. Since $\delta(G_1)+\delta(G_2)+\delta(G_3)>N$ and since $G_2\cap G_3=\emptyset$ by Lemma~\ref{lem: mindeg, no 23 edge r=2}, the pigeon-hole principle together with Lemma~\ref{lem:  intersections, if large, are non-empty} implies that for any $x\in X$ and $y \in Y$ we have $\vert N_{G_1 \cap G_2}(x) \vert>\frac{N}{6}$ and $\vert N_{G_1 \cap G_3}(y) \vert > \frac{N}{6}$ respectively. We gather some elementary observations about the structure of $G_2$ and $G_3$ before analysing their structure with respect to $X\sqcup Y$ in greater detail.
\begin{lemma}\label{lemma: no G3 edges between X and Y no triangles in G2Y or G3X} The following hold:
	\begin{enumerate}[(i)]
		\item $G_3[X, Y]=\emptyset$;
		\item$G_2[Y]$ is triangle-free;
		\item $G_3[X]$ is triangle-free.
	\end{enumerate}
\end{lemma}
\begin{proof}
\noindent  \textbf{Part (i):} suppose that there exist $x \in X$ and $y \in Y$ with $xy \in G_{3}$. Note that the sets $N_{G_1 \cap G_{2}}(x)$, $N_{G_{1}}(y)$ and $N_{G_2}(y)$ are disjoint sets. Since they have sizes at least $\delta^+(G_1\cap G_2)> \frac{N}{6}$, $\frac{N}{2}$ and $\frac{N}{3}$ respectively, this gives a contradiction.

\noindent  \textbf{Part (ii):} suppose that $y',y' \in Y$ and $v\in V$ are the vertices of a triangle $T$ in $G_{2}$. Since $\delta(G_1) + \delta(G_3) > \frac{2N}{3}$, by averaging there exists $w$ sending at least $3$ edges in colours $1$ or $3$ to $T$. Further, all three of these edges must be in same colour $c\in\{1,3\}$ (otherwise we have a rainbow triangle). It is easy to check that $N_{G_1 \cap G_3}(y)$, $N_{G_1 \cap G_3}(y')$  $N_{G_{4-c}}(v)$ and $N_{G_2}(w)$ are pairwise disjoint subsets of combined size at least $2\delta^+(G_1\cap G_3)+\delta(G_3)+\delta(G_2)>N$, a contradiction. Thus any triangle in $G_2$ must have at most one vertex in $Y$.

\noindent  \textbf{Part (iii):} suppose that $x\in X$ and $v,v'\in V$ are the vertices of a triangle in $G_3$. Then $N_{G_1\cap G_2}(x)$, $N_{G_2}(v)$ and $N_{G_1}(v')$ are pairwise disjoint sets of combined size at least $\delta^+(G_1\cap G_2)+\delta(G_2)+\delta(G_1)>N$, a contradiction. Thus any triangle in $G_3$ must be contained wholly inside $Y$.
\end{proof}
\noindent
Our next aim is to prove that $G_{3}[X]$ must be bipartite.
\begin{lemma} \label{lemma: G_3 bipartite r=2}
$G_{3}[X]$ is bipartite. 
\end{lemma}
\begin{proof}
Suppose that $G_{3}[X]$ is not bipartite, and let $C$ be a shortest odd cycle in $G_3[X]$. by Lemma~\ref{lemma: no G3 edges between X and Y no triangles in G2Y or G3X}(iii), this cycle $C$ must have length at least $5$. Let $x_{1}, \dots, x_{m}$ be the vertices of $C$ in cyclic order. Since $C$ has minimal length, $N_{G_3}(x_{i})$ and $N_{G_3}(x_{j})$ may intersect only when $i$ and $j$ differ by $2$ modulo $m$. We split the proof into two cases, depending on the size of $\delta(G_3)$.

\noindent \textbf{Case 1: $\delta(G_3)>\frac{N}{4}$.}
\begin{claim} \label{claim: not too many neighbours r=2}
Suppose that $\delta(G_3) \geq \frac{N}{4}$. Then every vertex $v\in V$ sends at most two edges of $G_2$ to $C$, and furthermore does not send edges in colour $2$ to adjacent vertices on the cycle.
\end{claim}
\begin{proof}
We use a similar argument to the proof of Claim~\ref{claim: not too many edges into cycle, r geq 3}.
Let $v \in V$, and suppose that $i, j$ and $k$ are distinct indices chosen so that $\left\{x_i, x_j, x_k\right\} \subseteq N_{G_2}(v)$. Since $C$ is a cycle of length at least $5$, two of these indices do not differ by exactly $2$ modulo $\vert C\vert $, say $i$ and $j$. Thus the sets $N_{G_3}(x_{i})$, $N_{G_3}(x_{j})$ and $N_{G_1}(v)$ are disjoint, contradicting the fact that $\delta(G_1) + 2\delta(G_3) > N$. The `furthermore' part of the claim follows immediately since cyclically adjacent indices do not differ by $2$ modulo $\vert C\vert$. 
\end{proof}
\noindent  Since every vertex $v\in V$ sends  at most two edges of $G_2$ into $C$, it follows by averaging over $x_i\in C$ that $\delta(G_2) \leq \frac{2N}{\vert C \vert}$. As $\delta(G_2) > \frac{N}{3}$, we must have $\vert C \vert = 5$.

We say that a pair $\left(i,i+2\right)$ is \emph{good} if $\vert N_{G_2}(x_i) \cup N_{G_2}(x_{i+2}) \vert < \frac{N}{2}$, with indices considered modulo $5$. Suppose that there exists $i$ which is contained in two good pairs. Since $\delta(G_2) > \frac{N}{3}$, we must have
\begin{align*}
\vert N_{G_2}(x_{i-2}) \cap N_{G_2}(x_{i}) \vert = \vert N_{G_2}(x_{i-2})\vert + \vert N_{G_2}(x_i)\vert - \vert N_{G_2}(x_i) \cup N_{G_2}(x_{i+2}) \vert &> \vert N_{G_2}(x_i)\vert -\frac{N}{6}\\
&>\frac{1}{2} \vert N_{G_2}(x_{i}) \vert 
\end{align*}
and similarly
$\vert N_{G_2}(x_{i+2}) \cap N_{G_2}(x_{i}) \vert > \frac{1}{2} \vert N_{G_2}(x_{i}) \vert$.
This implies that $N_{G_2}(x_{i-2}) \cap N_{G_2}(x_{i}) \cap N_{G_2}(x_{i+2}) \neq \emptyset$, contradicting Claim ~\ref{claim: not too many neighbours r=2}. Hence no vertex is contained in two good pairs, and (as the cycle $C$ has odd length) there exists a vertex, say $x_{1}$, that is not contained in any good pair.

Observe that for any vertices $v_1,v_2\in V$, if $w\in N_{G_3}(v_1)\cap N_{G_3}(v_2)$ then $N_{G_1}(w)$ is disjoint from $N_{G_2}(v_1)\cup N_{G_2}(v_2)$ (otherwise we have a rainbow triangle). Since $\delta(G_1)>\frac{N}{2}$, this implies in particular that if $\vert N_{G_2}(v_1)\cup N_{G_2}(v_2)\vert \geq\frac{N}{2}$, then $\vert N_{G_3}(v_1)\cap N_{G_3}(v_2)\vert =0$.  So the definition of a good pair above, together with the minimality of $C$, implies that $N_{G_3}(x_1)$ is disjoint from $\bigcup_{i=2}^5 N_{G_3}(x_i)$.


Now suppose that there exists $i \neq 1$ so that $N_{G_2}(x_1) \cap N_{G_2}(x_i)$ is non-empty, and let $w$ be an element contained in this set. Then $N_{G_1}(w)$, $N_{G_3}(x_1)$ and $N_{G_3}(x_i)$ are disjoint sets of sizes strictly greater than $\frac{N}{2}$, $\frac{N}{4}$ and $\frac{N}{4}$, which is a contradiction. 

Hence $N_{G_2}(x_1)$ is disjoint from $\bigcup_{i=2}^5N_{G_2}(x_i)$. By the second part of Claim~\ref{claim: not too many neighbours r=2}, this implies the sets $N_{G_2}(x_1)$, $N_{G_2}(x_2)$ and $N_{G_2}(x_3)$ are pairwise disjoint, each having size at least $\delta(G_2)>\frac{N}{3}$, a contradiction. Thus in the case $\delta(G_3) \geq \frac{N}{4}$, $G_{3}[X]$ must be bipartite.

\noindent \textbf{Case 2: $\delta(G_3) \leq \frac{N}{4}$.} In this case, our condition $\delta(G_2)+\delta(G_3)>\frac{2N}{3}$ implies $\delta(G_2) > \frac{5N}{12}$. We start by proving an analogue of Claim ~\ref{claim: not too many neighbours r=2}.
\begin{claim} \label{claim: not too many again r=2}
If $\vert C\vert \geq 7$, then every vertex $v\in V$ sends at most $4$ edges of $G_2$ into $C$.
\end{claim}
\begin{proof}
Set $\vert C\vert =2\ell+1$. Suppose for a contradiction that $v$ sends at least five edges of $G_2$ into $C$, and let $U \subseteq C$ be a set of five neighbours of $v$ in $G_2$ that lie on $C$.  Since $\vert C\vert \geq 7$, it is easily seen that $U$ must contain a set of three distinct vertices $\{x_i, x_j, x_k\}$ so that the distance between any two vertices in this set is not $2$. 
Now $N_{G_1}(v)$, $N_{G_3}(x_i)$, $N_{G_3}(x_j)$ and $N_{G_3}(x_k)$ are disjoint sets, contradicting the fact that $3\delta(G_3) + \delta(G_1) > N$.
\end{proof}
\noindent
Set $A = \left\{v \in V: N_{G_2}(v) \cap C \neq \emptyset \right\}$. If $\vert C\vert \geq 7$, then by Claim~\ref{claim: not too many again r=2} and averaging, we must have 
\begin{align}\label{eq: upper bound on A-set}
\vert A\vert & \geq \frac{\vert C\vert }{4} \delta(G_2).
\end{align}

Our next aim is to prove a statement analogous to Claim~\ref{claim: vertices send few edges to C} (from the proof of the $r\geq3$ case), with colours $2$ and $3$ interchanged due to our different set-up. 
\begin{claim}\label{claim: vertices send few edges to C r=2}
If $\vert C\vert \geq 7$, then for every $v\in V$, $\vert G_1[\{v\}, C]\vert + \vert G_2[\{v\}, C]\vert\leq \vert C\vert $. Moreover, this inequality is strict if $v$ sends an edge in colour $2$ to $C$.
\end{claim}
\begin{proof}
The proof follows is identical to the proof of Claim~\ref{claim: vertices send few edges to C}, apart from the very last assertion: at the end of the proof we have to verify that no vertex $v$ can send $\vert C \vert$ edges of colour $2$ to $C$. In the proof of Claim~\ref{claim: vertices send few edges to C}, we used Claim~\ref{claim: G3 nhoods indep in G2} to do this; to prove Claim~\ref{claim: vertices send few edges to C r=2}, we appeal instead to Claim~\ref{claim: not too many again r=2}. 
\end{proof}

\noindent 
Suppose that $\vert C\vert \geq 7$. Counting the edges in $G_{1}$ and $G_{2}$ with at least one endpoint in $C$, we get from combining  Claim~\ref{claim: vertices send few edges to C r=2} with the lower bound~\eqref{eq: upper bound on A-set} on $\vert A\vert$ that
\begin{align*}
\vert C \vert\left(\delta(G_1) + \delta(G_2)\right)\leq \left(N - \vert A \vert\right) \vert C \vert + \vert A \vert \left(\vert C \vert -1\right) = \vert C\vert N -\vert A\vert \leq \vert C \vert \left(N-\frac{\delta(G_2)}{4}\right) .
\end{align*}
This is a contradiction since $\delta(G_1) + \frac{5}{4} \delta(G_2) >\frac{N}{2}+\frac{25N}{48}> N$. Thus we must have $\vert C \vert = 5$.

Consider the set $B = \left\{v \in V: N_{G_1 \cap G_2}(v) \cap C \neq \emptyset\right\}$. As in the 
proof of Claim~\ref{claim: vertices send few edges to C},  we can deduce that a vertex $b \in B$ can send a total of at most $4$ edges in colours $1$ or $2$ into $C$. Recall $C\subseteq X$, so $\delta^+(G_1\cap G_2)>\frac{N}{6}$, and by averaging over vertices of $C$ we have $\vert B \vert \geq \frac{5\delta^+(G_1\cap G_2)}{2} > \frac{5N}{12}$. Again, double-counting the edges in $G_{1}$ and $G_{2}$ with at least one endpoint in $C$, we obtain that
\begin{align*}
 5 \left(\delta(G_1)+ \delta(G_2)\right) \leq 5 \left(N - \vert B \vert\right) + 4 \vert B \vert = 5N-\vert B\vert.
\end{align*}
Since $\delta(G_2) > \frac{5N}{12}$, this implies $\vert B \vert < \frac{5N}{12}$, contradicting our earlier lower bound on $\vert B\vert$. This completes the proof that $G_{3}[X]$ must be bipartite.
\end{proof}

\begin{lemma}\label{lemma: if X empty, G2 bip}
If $X=\emptyset$, then $G_2=G_2[Y]$ is bipartite.
\end{lemma}
\begin{proof}
Assume $X=\emptyset$ (and thus $Y=V$). Suppose for a contradiction that $G_2$ is not bipartite, and let $C$ be a shortest odd cycle in $G_2$. By Lemma~\ref{lemma: no G3 edges between X and Y no triangles in G2Y or G3X}, $G_{2}$ is triangle-free, and so $\vert C\vert \geq 5$. Further, since $G_2$ is not bipartite, the Andrasfai--Erd{\H o}s --S\'os theorem~\cite{AndrasfaiErdosSos74} implies that $\delta(G_2) \leq \frac{2N}{5}$ and hence that $\delta(G_3)> \frac{4N}{15}$. As $C$ is a shortest cycle, every vertex $v\in V$ can send at most two edges of $G_2$ into $C$. Thus in particular $\vert C\vert \delta(G_2)\leq 2N$, and since $\delta(G_2)>\frac{N}{3}>\frac{2N}{7}$ we must have $\vert C\vert =5$.

Let $y_1, \ldots y_5$ denote the vertices of $C$. Considering indices modulo $5$, we have that for every $i\in [5]$, $N_{G_1}(y_i)$ is disjoint from $N_{G_3}(y_{i-1})\cup N_{G_3}(y_{i+1})$ (otherwise we have a rainbow triangle). Since $\delta(G_1)+2\delta(G_3)>\frac{31N}{30}$, this implies that there exist vertices $v_i$, $i\in [5]$, with $v_iy_{i-1}, v_iy_{i+1}\in G_3$ for all $i$.

Now $N_{G_1}(v_4)$ is disjoint from $N_{G_2}(y_3)\cup N_{G_2}(y_{5})$, whence we have $\vert N_{G_2}(y_3)\cup N_{G_2}(y_{5})\vert<\frac{N}{2}$ and 
\begin{align*}\vert N_{G_2} (y_{3})\cap N_{G_2} (y_{5})\vert \geq \vert N_{G_2}(y_3)\vert +\delta(G_2) -\vert N_{G_2}(y_3)\cup N_{G_2}(y_{5})\vert >\vert N_{G_2}(y_3)\vert  -\frac{N}{6}> \frac{1}{2}\vert N_{G_2}(y_3)\vert .\end{align*}
Similarly, we have $\vert N_{G_2}(y_1)\cap N_{G_2}(y_3)\vert >\frac{1}{2}\vert N_{G_2}(y_3)\vert$. However this implies that the triple intersection $N_{G_2}(y_1)\cap N_{G_2}(y_3)\cap N_{G_2}(y_5)$ is non-empty. In particular $y_1$ and $y_5$ have a common neighbour in $G_2$, contradicting the fact that $G_2$ is triangle-free. Thus $G_2$ must be bipartite as claimed.
\end{proof}

\noindent We now show  $X =V$, i.e.\ that every vertex is incident with 
bi-chromatic edges in colours $12$. 
\begin{lemma}\label{lemma: all vertices incident with 12 edges, r=2}
	$X = V$. 
\end{lemma}
\begin{proof}
	\noindent We first show $X$ is non-empty.
	\begin{claim}\label{claim: r=2, X nonempty}
		$X\neq \emptyset$.
	\end{claim}	
	\begin{proof}	
		Suppose $X=\emptyset$ (and thus $Y=V$). By Lemma~\ref{lemma: if X empty, G2 bip}, $G_2=G_2[Y]$ is bipartite. Let $V = V_{1} \sqcup V_{2}$ be a bipartition of $V$ such that $G_2 \subseteq \left(V_{1},V_{2}\right)^{(2)}$.

		We claim that every vertex sends edges of $G_3$ to at most one of $V_1$, $V_2$. Indeed, suppose this is not the case, and there exist $u,v$ and $w$ with $u,v \in V_{i}$ and $w \in V_{3-i}$ with $uv,uw \in G_{3}$. Then $N_{G_1}(u)$, $N_{G_2}(v)$ and $N_{G_2}(w)$ are disjoint sets (as the latter two are contained in distinct vertex classes of the bipartition), which contradicts the fact that $\delta(G_1) + 2\delta(G_2) > N$.

		Assume without loss of generality that $\vert V_{1} \vert \geq \frac{N}{2}$. Suppose that there exists $u \in V_{1}$ with $N_{G_3}(u) \subseteq V_{2}$. Since $\vert V_{1} \vert \geq \frac{N}{2}>N-\delta(G_1)$, there exists $v \in V_{1}$ with $uv \in G_{1}$. As we have no rainbow triangles,  $N_{G_{3}}(u)$ and $N_{G_{2}}(v)$ are disjoint subsets of $V_{2}$, which contradicts the fact that $\vert V_{2} \vert\leq \frac{N}{2}  < \delta(G_2) + \delta(G_3)$. Hence no $v \in V_{1}$ can send an edge in colour $3$ to $V_{2}$.

		Finally, given $u \in V_{2}$, there exists $v \in V_{1}$ with $uv \in G_{1}$ (since $\delta(G_1)>\frac{N}{2}\geq \vert V_2\vert$). As we have no rainbow triangles, $N_{G_{3}}(u)$ and $N_{G_{2}}(v)$ are disjoint subsets of $V_{2}$, contradicting the fact that $\vert V_{2} \vert< \delta(G_2) + \delta(G_3)$, which completes the proof. 
	\end{proof}

\noindent Let $X = A \sqcup B$ be a bipartition of $G_{3}[X]$ (which exists by Lemma~\ref{lemma: G_3 bipartite r=2}).
\begin{claim}\label{claim: send edge in colours 1 or 2 to at most two of A, B,Y}
$\forall v \in V$ and $i \in \left\{1,2\right\}$, $v$ sends an edge of $G_i$ to at most two of the sets $A$, $B$ and $Y$. 
\end{claim}
\begin{proof}
Suppose for a contradiction that there exist $v \in V$ and vertices $a \in A$, $b \in B$ and $y \in Y$ with $\left\{a,b,y\right\} \subseteq N_{G_i}(v)$. Then the sets $N_{G_3}(a)$, $N_{G_3}(b)$ and $N_{G_3}(y)$ are disjoint, as they are subsets of $B$, $A$ and $Y$ respectively. Furthermore, they are pairwise disjoint with $G_{3-i}(v)$, which contradicts the fact that $\min\left(\delta(G_2),\delta(G_1)\right) + 3\delta(G_3) > \frac{2N}{3} + 2\delta(G_3) > N$ (since $\delta(G_3)> \frac{N}{6}$ by~\eqref{eq: ub and lb bound on delta(G_i)}). The claim follows. 
\end{proof}
\noindent
Since $\delta(G_1) > \frac{N}{2}$, there must be an edge of $G_1$ between $Y$ and $X=A \cup B$. Assume without loss of generality that $a\in A$ and $y\in Y$ are such that $ay\in G_1$. As we observed in the paragraph above Lemma~\ref{lemma: no G3 edges between X and Y no triangles in G2Y or G3X}, $y$ is incident with at least one bi-chromatic edge in colours $13$, which by Lemma~\ref{lemma: no G3 edges between X and Y no triangles in G2Y or G3X}(i) must be to some vertex $y'\in Y$. By Claim~\ref{claim: send edge in colours 1 or 2 to at most two of A, B,Y}, this implies $y$ sends no edge in colour $1$ to the set $B$.  Further, since $y\in Y$, there is no bi-chromatic edge in colours $12$ incident with $y$, and its neighbourhoods in $G_1$ and $G_2$ are disjoint

Thus $N_{G_1}(y)\subseteq V\setminus B$, $N_{G_3}(a)\subseteq B$ and $N_{G_2}(y)$ are disjoint subsets (otherwise we have a rainbow triangle or a bi-chromatic edge in colours $12$ incident with $y$), which contradicts the fact that $\delta(G_1) + \delta(G_2) + \delta(G_3) > N$. The lemma follows. 
\end{proof}

\noindent By Lemma~\ref{lemma: all vertices incident with 12 edges, r=2}, $X=V$ (i.e.\ every vertex is incident with at least $N/6$ bi-chromatic edges in colours $12$), and by Lemma~\ref{lemma: G_3 bipartite r=2}, $G_3=G_3[X]$ is a bipartite graph with bipartition $A\sqcup B$. Assume without loss of generality that $\vert A\vert \geq \vert B\vert $. Our next aim is to verify that every vertex sends an edge in colour $1$ to both of $A$ and $B$.

\begin{claim}\label{1neigh in both}
For every $v \in V$ there exist $a \in A$ and $b \in B$ with $va, vb \in G_1$. 
\end{claim}
\begin{proof}
Fix $v\in V$. Since $\vert A \vert \geq \frac{N}{2}>N-\delta(G_1)$, it is clear that there exists $a\in A$ with $av\in G_1$. Suppose for a contradiction that $N_{G_1}(v) \subseteq A$.  Let $u$ be chosen so that $uv \in G_3$, and consider $N_{G_2}(u) \cap B$. Since $N_{G_2}(u) \cap A$ and $N_{G_1}(v)$ are disjoint subsets of $A$, it follows that 
\begin{align*}
\vert N_{G_2}(u) \cap B \vert \geq \delta(G_2)-\vert N_{G_2}(u)\vert\geq \delta(G_2) - \left(\vert A \vert-\delta(G_1)\right) > \vert B \vert - \delta(G_3),
\end{align*}
as we have $\delta(G_1) + \delta(G_2) + \delta(G_3) > N=\vert A\vert +\vert B\vert$. Thus for any $w \in A$ we must have $N_{G_3}(w) \cap \left(N_{G_2}(u) \cap B\right) \neq \emptyset$, which implies that $uw \not \in G_{1}$. Hence $N_{G_1}(u) \subseteq B$, which  contradicts the fact that $\delta(G_1)>\frac{N}{2}\geq \vert B\vert$. The claim follows. 
\end{proof}
\begin{claim}\label{cl: r= 2 case, delta G3 <N/4} $\delta(G_3)<\frac{N}{4}$.
\end{claim}
\begin{proof}
Suppose $\delta(G_3)\geq \frac{N}{4}$.	Fix $a\in A$. Since $G_2 \cap G_3 = \emptyset$ and $\vert B \vert < \delta(G_2) + \delta(G_3)$, there exists $a' \in A$ with $aa' \in G_2$. Suppose there exists $b\in B$ with $ab\in G_2$. Then $N_{G_1}(a)$, $N_{G_3}(a') \subseteq B$ and $N_{G_3}(b) \subseteq A$ are disjoint sets, which contradicts the fact that $\delta(G_1) + 2\delta(G_3) > N$ whenever $\delta(G_3) \geq \frac{N}{4}$. Hence no such $b\in B$ can exist, and we must have $\vert G_2[A,B]\vert = 0$.

Since $\delta(G_2) > \frac{N}{3}$, this implies $\vert B \vert > \frac{N}{3}$. For $b \in B$, there exists $a \in A$ with $ab \in G_1$ as $\delta(G_1) > \frac{N}{2}$. As $\mathbf{G}$ does not contain rainbow triangles, the sets $N_{G_2}(a)$ and $N_{G_3}(b)$ are disjoint subsets of $A$. Hence $\vert A \vert > \frac{2N}{3}$, which contradicts the fact that $\vert B \vert > \frac{N}{3}$. Thus we must have $\delta(G_3)<\frac{N}{4}$, as claimed.
\end{proof}
\noindent By Claim~\ref{cl: r= 2 case, delta G3 <N/4},  $\delta(G_3) < \frac{N}{4}$, whence by our minimum degree assumption $\delta(G_2) > \frac{5N}{12}$. 
\begin{claim}\label{joint 3's}
Let $u,v \in V$ be vertices both in $A$ or both in $B$. Then there exists $w$ such that $uw, vw \in G_3$. 
\end{claim}
\begin{proof}
Suppose that $u, v \in V$ are vertices from the same vertex-class of the bipartition $A\sqcup B$, and for simplicity assume that $u,v \in A$. (Our argument will not rely on the relative sizes of $A$ and $B$, so we may make such an assumption without loss of generality.) Since $\delta(G_1) > \frac{N}{2}$, there exists $w \in V$ so that $uw, vw \in G_1$. Furthermore, Claim~\ref{1neigh in both} implies that there exists $b \in B$ with $wb \in G_1$. 

Suppose that $N_{G_3}(u)$ and $N_{G_3}(v)$ are disjoint. Since $\mathbf{G}$ does not contain a rainbow triangle, the sets $N_{G_2}(w)$, $N_{G_3}(b)\subseteq A$ and $N_{G_3}(u)\cup N_{G_3}(v)\subseteq B$ are disjoint. This contradicts the fact that 
\begin{align*}
\vert N_{G_2}(w)\vert + \vert N_{G_3}(b)\vert + \vert N_{G_3}(u)\cup N_{G_3}(v)\vert \geq \delta(G_2) + 3\delta(G_3) > \frac{2N}{3} + 2\frac{N}{6} = N. 
\end{align*}
The claim follows. 
\end{proof}
\noindent 
Fix $b \in B$, and consider the set $C := N_{G_2}(b) \cap A$. By Claim~\ref{1neigh in both}, there exists $a \in A$ with $ab \in G_1$. Since $N_{G_3}(a)$ and $N_{G_2}(b) \cap B$ are disjoint subsets of $B$, it follows that $\vert C \vert \geq \delta(G_2) + \delta(G_3) - \vert B \vert > \frac{N}{6}$. Our next aim is to prove by a simple averaging argument that there exists a vertex $c\in C$ which sends many edges in colour $2$ to $B$, which we will use to derive a final contradiction and conclude the proof of Theorem~\ref{theorem: r =2 case}. 
\begin{claim}
There exists $c \in C$ with $\vert N_{G_2}(c) \cap B \vert > \vert B \vert - \delta(G_3).$
\end{claim}
\begin{proof}
We first prove that for $b' \in B$ we have $\vert N_{G_2}(b') \cap C \vert > \vert C \vert  - \left(\frac{N}{2} - \delta(G_2)\right)$. Indeed, suppose that there exists a vertex $b' \in B$ which violates this condition. Let $a \in A$ be chosen so that $ab, ab' \in G_3$, which exists by Claim~\ref{joint 3's}. 

Since $b'$ satisfies the condition $\vert N_{G_2}(b') \cap C \vert \leq \vert C \vert  - \left(\frac{N}{2} - \delta(G_2)\right)$, we obtain that 
\begin{align*}
\vert N_{G_2}(b) \cup N_{G_2}(b') \vert \geq \vert C \cup N_{G_2}(b')\vert \geq \vert N_{G_2}(b') \vert + \vert C \vert - \vert N_{G_2}(b') \cap C \vert \geq \frac{N}{2}>N-\vert N_{G_1}(a)\vert.
\end{align*}
However, this contradicts the fact that the sets $N_{G_2}(b) \cup N_{G_2}(b')$ and $N_{G_1}(a)$ are disjoint. Hence for any $b' \in B$ we have $\vert N_{G_2}(b') \cap C \vert > \vert C \vert  - \left(\frac{N}{2} - \delta(G_2)\right)$. By averaging, it follows that there exists a vertex $c \in C$ with 
\begin{align} \label{eq:lowerbound}
\vert N_{G_2}(c) \cap B \vert > \frac{\left(\vert C \vert + \delta(G_2) - \frac{N}{2}\right) \vert B \vert}{\vert C \vert}.
\end{align}
Now,
\begin{align} \label{eq:intermediate}
\left(\vert C \vert + \delta(G_2) - \frac{N}{2}\right) \vert B \vert - \left(\vert B \vert - \delta(G_3) \right) \vert C \vert 
= \left(\delta(G_2) - \frac{N}{2}\right) \vert B \vert + \delta(G_3) \vert C \vert.
\end{align}
Since $\frac{5N}{12}<\delta(G_2) < \frac{N}{2}$, $\vert B \vert \leq \frac{N}{2}$ and $\vert C \vert > \frac{N}{6}$, the right-hand side of~\eqref{eq:intermediate} is at least 
\begin{align*}
\delta(G_2)\frac{N}{2} - \frac{N^2}{4} + \delta(G_3)\frac{N}{6} 
\geq \frac{2N}{3} \cdot \frac{N}{6} + \frac{5N}{12} \cdot \frac{N}{3} - \frac{N^2}{4} = 0.
\end{align*}
Thus combining equations~\eqref{eq:lowerbound} and~\eqref{eq:intermediate}, we obtain that $\vert N_{G_2}(c) \cap B \vert > \vert B \vert - \delta(G_3)$. 
\end{proof}
\noindent
We are ready ar last to complete the proof of Theorem~\ref{theorem: r =2 case}. Let $c \in C \subseteq A$ be chosen so that $\vert N_{G_2}(c) \cap B \vert > \vert B \vert - \delta(G_3)$. Let $a \in A$ be chosen so that $ac \in G_1$ (such an $a$ exists since $\vert A \vert \geq \frac{N}{2}>N-\delta(G_1)$). As $\mathbf{G}$ does not contain a rainbow triangle, the sets $N_{G_2}(c) \cap B$ and $N_{G_3}(a)$ must be disjoint subsets of $B$. However, this contradicts the fact that $\vert N_{G_3}(a)\vert + \vert N_{G_2}(c) \cap B \vert > \vert B \vert$. This final contradiction completes the proof. 
\end{proof}

\section{Concluding remarks}\label{section: concluding remarks}	
Our work leaves a number of questions open. To begin with, one could ask for a generalisation of Theorem~\ref{theorem: forcing min degrees} in the style of~\cite[Theorem 1.13]{FRMarkstromRaty22}. Explicitly, say that a triple $(\delta_1, \delta_2, \delta_3)$ with $\delta_1\geq \delta_2\geq \delta_3$ is a \emph{mindegree forcing triple} if for all $n$ sufficiently large, every $n$-vertex $3$-colouring template $\mathbf{G}$ with $\delta(G_i) > \delta_i n$ for every $i\in [3]$ must contain a rainbow triangle.
\begin{problem}
Determine the set of mindegree forcing triples.
\end{problem}
\noindent As a slightly easier question, we ask whether the bound in Theorem~\ref{theorem: forcing min degrees} can be significantly improved  if $\max\left(\delta(G_2), \delta(G_3)\right) > \frac{N}{r+1}$:
\begin{question} \label{stability} Fix $r\in \mathbb{Z}_{\geq 2}$. Let $\mathbf{G}$ be an $n$-vertex Gallai $3$-colouring template with $\delta(G_1)\geq \delta(G_2)\geq \delta(G_3)>0$. Suppose $\delta(G_1)>\frac{r-1}{r}n$ and $\delta(G_2) > \frac{n}{r+1}$. Does there exist a constant $\varepsilon=\varepsilon(r)>0$ such that $\delta(G_2)+\delta(G_3) \leq \left(1-\varepsilon \right)\frac{2n}{r+1}$?
\end{question}
\noindent If such constants $\varepsilon(r)$ exist, it would naturally be interesting to determine their optimal value for each $r\in\mathbb{Z}_{\geq 2}$.   We expect that the proof of Theorem~\ref{theorem: forcing min degrees} could be adapted to yield stability versions of our results (albeit with an unusual flavour: for each $r$,we expect near-extremal constructions must be close to one of the $\mathbf{C_{P,c}}$ templates given in Construction~\ref{construction: general}), but we had not explored this further due to the length of the paper. We suspect such stability results may be a useful tool for tackling Question~\ref{stability}. 

In a different direction, there has been considerable interest in rainbow versions of Dirac's theorem since the work of Joos and Kim on the subject~\cite{JoosKim20} (see notably~\cite{GuptaHamannMuyesserParczykSgueglia22} for some very recent progress on these questions). One could consider analogues of our results in this setting. Explicitly, suppose one has an $n$-vertex  $n$-colouring template $\mathbf{G}^{(n)}$. Let $\mathbf{\delta}:=\frac{1}{n}\left(\delta(G_1), \delta(G_2), \ldots ,  \delta(G_n)\right)$. Which vectors $\mathbf{\delta}$ guarantee the existence of a Hamilton cycle? Previous results have focussed on the value of $\min_i \delta(G_i)$ required for this. But perhaps if many of the $\delta(G_i)$ are large, then we can afford for the smallest one to be quite a bit smaller. For example, suppose we know $\delta(G_1)\geq \delta(G_2)\geq \delta(G_k)=\alpha n$ and $\delta(G_{k+1})\geq \delta(G_{k+2})\geq \ldots \delta(G_n)\geq \beta n$. For a given $k<n$, what values of $\alpha\geq \beta>0$ guarantee the existence of a rainbow Hamilton cycle?

Another natural question in the style of Joos and Kim, given that the existence of rainbow triangles is now fairly well understood, is the existence of triangle factors in colouring templates.
\begin{question}
Let $\mathbf{G}^{(3)}$ be an $3$-colouring template on $3n$ vertices. What triples $(\delta(G_1), \delta(G_2), \delta(G_3))$ guarantee the existence of $n$ vertex-disjoint rainbow triangles in $\mathbf{G}$?
\end{question}

Finally, we focused in this work on colouring \emph{templates}, in which colour classes may overlap. Following Erd{\H o}s and Tuza~\cite{ErdosTuza93}, one could instead consider problems for \emph{colourings} of $K_n$ or of subgraphs of $K_n$, in which the colour classes must be disjoint. Can one obtain analogues of Theorem~\ref{theorem: forcing min degrees} in these settings?

In particular a conjecture of Aharoni~\cite{AharonDeVosHolzman19} (which implies the famous Cacceta-H\"aggkvist conjecture) states that every $n$-vertex $n$-coloured graph in which each colour class has size at least $r$ has \emph{rainbow girth} at most $\lceil n/r\rceil$ (i.e.\ contains a rainbow cycle of length at most $\lceil n/r\rceil$); see Hompe and Huynh~\cite{HompeHuynh22} for recently announced progress on this problem.  Motivated by Aharoni's rainbow cycle conjecture, Clinch, Goerner, Huynh and Illingworth~\cite[Question 2.5]{ClinchGoernerHuynhIllingworth} proposed the more general problem of finding the maximum value of the rainbow girth in an $n$-vertex $t$-colouring template in which every colour class has size at least $r$, while Hompe, Qu and Spirkl~\cite{HompeQuSpirkl} studied the set of pairs $(\alpha, \beta)$ such that any $\alpha n$-coloured graph in which every colour class has size at least $\beta n$ must contain a rainbow triangle. In a similar spirit to the results in this paper, one could consider variants of Aharoni's conjecture of the following form: given a lower bound on the rainbow girth, how large could (a) the sum of the colour classes, or (b) the minimum degree across the colour classes
 be in an $n$-vertex $t$-colouring template? 
\nocite{*}
\section*{Acknowledgements}
This research was supported by the Swedish Research Council grant VR 2021-03687 and postdoctoral grant 213-0204 from the Olle Engkvist Foundation.

\end{document}